\documentclass[review]{elsarticle}

\usepackage{lineno,hyperref}
\usepackage[utf8]{inputenc}

\usepackage{amssymb}
\usepackage{amsmath}
\usepackage{mathrsfs}

\usepackage{comment}
\usepackage{color}

\newtheorem{theorem}{Theorem}[section]
\newtheorem{corollary}[theorem]{Corollary}
\newtheorem{definition}[theorem]{Definition}

\newtheorem{lemma}[theorem]{Lemma}
\newtheorem{proposition}[theorem]{Proposition}
\newtheorem{remark}[theorem]{Remark}
\newtheorem{hypothesis}[theorem]{Hypotheses}
\newenvironment{proof}[1][Proof]{\noindent\textbf{#1.} }{\ \rule{0.5em}{0.5em}}

\numberwithin{equation}{section}

\newcommand{\N}{\mathbb{N}}
\newcommand{\Z}{\mathbb{Z}}

\newcommand{\R}{\mathbb{R}}
\newcommand{\hR}{\,\!^\ast\R}

\renewcommand{\Lambda}{\mathbb{X}}

\newcommand{\dt}{\varepsilon_t}
\newcommand{\dx}{\varepsilon}
\newcommand{\x}{x}

\newcommand{\ns}[1]{\,\!^\ast#1} %comandi sinistri
\newcommand{\sh}[1]{\,\!^\circ#1} %comandi sinistri
\newcommand{\D}{\mathbb{D}}
\newcommand{\vsum}[1]{||#1||_1}
\newcommand{\test}{\mathscr{D}_\Lambda}
\newcommand{\tests}{\mathscr{D}}
\newcommand{\grid}[1]{\mathbb{G}({#1})}
\newcommand{\dom}{\Omega_{\Lambda}}

\newcommand{\simplex}{\mathbb{S}}
\newcommand{\steady}{\widetilde{u}}
\newcommand{\eq}{\steady_h}

\newcommand{\stdiv}{\mathrm{div}}
\renewcommand{\div}{\stdiv_{\Lambda}}
\newcommand{\stgrad}{\nabla}
\newcommand{\grad}{\nabla_\Lambda}

\newcommand{\lap}{\Delta_\Lambda}

\renewcommand{\sim}{\approx}
\newcommand{\stable}{$\ns{}$stable}
\newcommand{\unstable}{$\ns{}$unstable}
\newcommand{\attractive}{$\ns{}$attractive}

\newcommand{\B}{\test'({\dom})}

\newcommand{\fin}{\hR_{fin}}
\renewcommand{\Lambda}{\mathbb{X}}
\newcommand{\leb}{\mu_L}

\newcommand{\supp}{\mathrm{supp\,}}

\newcommand{\normal}{\hat{n}}
\renewcommand{\sim}{\approx}

\newcommand{\norm}[1]{\left\lVert#1\right\rVert}
\newcommand{\bcf}{C^0_b}
\newcommand{\rad}{\mathbb{M}}
\newcommand{\prob}{\rad^{\mathbb{P}}}

\newcommand{\ldual}{\langle}
\newcommand{\rdual}{\rangle_{\tests'(\Omega)}}
\newcommand{\lp}[1]{\widehat{#1}}

\begin{document}

\begin{frontmatter}
	
	\title{A grid function formulation of a class of ill-posed parabolic equations}

	\author{Emanuele Bottazzi}
	\address{University of Pavia, Italy}
	\ead[url]{https://sites.google.com/view/emanuele-bottazzi/home}
	
	
	%\address[mymainaddress]{University of Trento, Italy}
	
	\begin{abstract}
		We study an ill-posed forward-backward parabolic equation
		with techniques of nonstandard analysis.
		Equations of this form arise in many applications, ranging from the description of phase transitions, to the dynamics of aggregating populations, and to	image enhancing.
		%It is well-known that, if $\phi \in C^1(\mathbb{R})$ and $\phi'(x) < 0$ for all $x \in (a,b)$, the aforementioned problem is ill-posed and only has measure-valued solutions.
		The equation is ill-posed in the sense that it only has measure-valued solutions that in general are not unique.
		By using grid functions of nonstandard analysis, we derive a continuous-in-time and discrete-in-space formulation for the ill-posed problem. This nonstandard formulation is well-posed and formally equivalent to the classical PDE, and has a unique grid solution that satisfies the properties of an entropy measure-valued solution for the original problem.
		By exploiting the strength of the nonstandard formulation, we are able to characterize the asymptotic behaviour of the grid solutions and to prove that they satisfy a conjecture formulated by Smarrazzo for the measure-valued solutions to the ill-posed problem.
		%We conclude with a brief discussion of the Riemann problem, from which we deduce that the grid solutions to problem \ref{abstract} feature a hysteresis loop.
	\end{abstract}
	
	
\end{frontmatter}

\section{Introduction}\label{section illposed}

In this paper, we will discuss a grid function formulation of the Neumann initial value problem
\begin{equation}\label{mark}
\left\{
\begin{array}{l}
\partial_t u=\Delta \phi(u) \text{ in } \Omega\\
\frac{\partial \phi(u)}{\partial \normal} = 0 \text{ in } [0,T] \times \partial\Omega\\
u(0,x) = u_0(x)
\end{array}\right.
\end{equation}
where $\Omega \subseteq \R^k$ is an open, bounded, and connected set with a smooth boundary, and $\normal$ is its unit exterior normal vector; %where $u: \R \times \Omega \rightarrow \R$
where $u_0 \in L^\infty(\Omega)$ satisfies $u_0(x) \geq 0$ for all $x \in \Omega$; and where the following hypotheses over $\phi$ are assumed:
\begin{hypothesis}\label{ipotesi di partenza}
	$\phi:\R\rightarrow\R$ satisfies:
	\begin{itemize}
		\item $\phi \in C^1(\R)$;
		\item $\phi(x) \geq 0$ for all $x \geq 0$ and $\phi(0) = 0$;
		\item $\phi$ is non-monotone, i.e.\ there are $u^-, u^+$ with $0 < u^- < u^+ \leq +\infty$ such that $\phi'(u) > 0$ if $u \in (0,u^-)\cup(u^+,+\infty)$ and $\phi'(u) < 0$ for $u \in (u^-,u^+)$;
		\item if $u^+ = +\infty$, then $\lim_{x \rightarrow +\infty} \phi(x) =0$.
	\end{itemize}
\end{hypothesis}

Forward-backward parabolic equations like \ref{mark} or the closely related
\begin{equation} \label{mark II}
u_t = \stdiv \phi(\stgrad u)
\end{equation}
with non-monotone $\phi$ have been used to describe various physical phenomena.
Cubic-like functions with $u^+ < +\infty$ arise for instance in models of phase transitions: in this context, the function $u$ represents the enthalpy and $\phi(u)$ the temperature distribution.
Equation \ref{mark} can be deduced as a consequence of the Fourier law.
If $u^+ = +\infty$, then equation \ref{mark} has been used in models of the dynamics of aggregating populations both in a discrete approximation (see for instance Horstmann, Othmer and Painter \cite{reinforcement} and Lizana and Padron \cite{sd}) and as a continuous diffusion approximation \cite{sobolev padron}.
Equation \ref{mark II} has been used to describe shearing of granular media (see Witelski, Shaeffer and Shearer \cite{discrete 2}); it is also noteworthy to mention that the Perona-Malik edge-enhancement algorithm via backward diffusion \cite{perona malik} is based on equation \ref{mark II}.

It is well-known that, if $u_0 \in L^{\infty}(\Omega)$ and $\mathrm{ess}\sup \phi(u) \leq u^-$ or $\mathrm{ess}\inf \phi(u) \geq u^+$, then the dynamics described by equation \ref{mark} amount to a parabolic smoothing.
The main feature of problem \ref{mark} is that it is ill–posed forward in time for $u$ in the interval $(u^-, u^+)$:
under this assumption, there are no weak solutions to problem \ref{mark}. % and, if we allow for measure-valued solutions, then solutions to problem \ref{mark} exist for any initial data but are not unique \cite{smarrazzo, matete, plotnikov}.% (however see MATETE).
%
%The hypothesis that $\phi$ is non-monotone is crucial both for the applications and for the description of the physical phenomena.
%For this reason, suitable approximations of initial value problems for equations \ref{mark} and \ref{mark II} has been studied in a variety of ways. %, especially under the hypothesis that $u^+ < +\infty$.
% (but see \cite{sobolev padron} for an example where $u^+ = +\infty$)
%For a discussion of these approaches, we refer to \cite{matete} and to \cite{smarrazzo}.
%
However, under the hypothesis that $0 < u^+ < +\infty$, it is possible to find a solution to problem \ref{mark} in the sense of Young measures by studying the limit of a suitable regularized problem \cite{plotnikov}.
If $u^+ = +\infty$, then the problem has no solution also in the class of Young measures; however, a notion of measure-valued solution has been given by Smarrazzo \cite{smarrazzo} and can be characterized as the sum of a Young measure and a Radon measure.
In Section \ref{section meassol}, we will recall the definition of the measure-valued solutions for these problems.

Despite the different notions of measure-valued solutions, that depend upon the value of $u^+$, we will be able to give a uniform formulation for the class of ill-posed PDEs by using grid functions of nonstandard analysis.
In Section \ref{section gridfz} we will review the properties of grid functions that will be used throughout the paper, and in Section \ref{ns models} we will derive the grid function formulation of problem \ref{mark} from a discrete model of diffusion.
In Section \ref{coherence}, we will show the relations between the solutions to the grid function formulation and the solutions to the original problem \ref{mark}.
In particular, if the solutions to the nonstandard problem are regular enough, they induce solutions to problem \ref{mark} that are coherent with the approaches of Plotnikov and Smarrazzo.
In Section \ref{properties} we will discuss the asymptotic behaviour of the grid solutions to problem \ref{mark} by studying the asymptotic behaviour of the solution to the grid function formulation.
We will also give a positive answer to a conjecture by Smarrazzo on the coarsening of the solutions to problem \ref{mark} when $u^+=+\infty$.
The paper concludes with a brief discussion of some properties of the grid solution to problem \ref{mark} with Riemann initial data.
In particular, by studying this initial value problem, we will show that the grid solution to problem \ref{mark} features a hysteresis loop, in agreement with the behaviour of the Young measure solution.

\subsection{Notation}

If $f : \R \rightarrow \R$, we will denote the derivative of $f$ by $f'$ or $D f$.
If $f : [0,T]\times \Omega \rightarrow \R$, we will think of the first variable of $f$ as the time variable, denoted by $t$, and we will write $f_t$ for the derivative $\frac{\partial f}{\partial t}$.
We adopt the multi-index notation for partial derivatives and, if $\alpha$ is a multi-index, we will denote by $D^\alpha f$ the function
%\begin{equation}\label{partialdev}
$$
D^\alpha f = \frac{\partial^{|\alpha|} f}{\partial x_1^{\alpha_1} \partial x_2^{\alpha_2}\ldots \partial x_k^{\alpha_k}}.
$$
%\end{equation}

We recall that, if $V$ is a real vector space, an element of $V'$ is a continuous linear functional $T : V\rightarrow\R$.
If $T \in V'$ and $\varphi \in V$, we will denote the action of $T$ over $\varphi$ by
$\langle T, \varphi\rangle_V$.

We will often reference the following real vector spaces:
\begin{itemize}
	\item $\bcf(\Omega) = \{ f \in C^0(\Omega) : f \text{ is bounded and } \lim_{|x|\rightarrow \infty} f(x) = 0 \}$.
	\item $C^l(\Omega) = \{ f \in C^0(\Omega) : f^\alpha \in C^0(\Omega)$ for any multi-index $\alpha$ with $\max{\alpha}\leq l\}$.
	\item $\tests(\Omega) = \{ f \in C^\infty(\Omega) : \supp f \subset \subset \Omega\}$.
	\item $\tests'(\Omega)$ is the space of real distributions over $\Omega$.
	Recall that, if $T \in \tests'(\Omega)$ and $\alpha$ is a multi-index, $D^\alpha T \in \tests'(\Omega)$ is defined by:
	$$
	\ldual D^\alpha T, \varphi\rdual = (-1)^{|\alpha|} \ldual T, D^\alpha \varphi\rdual.
	$$
	\item In the sequel, measurable will mean measurable with respect to $\leb$, the Lebesgue measure over $\R^n$.
	Consider the equivalence relation given by equality almost everywhere: two measurable functions $f$ and $g$ are equivalent if $\leb(\{x\in\Omega : f(x) \not = g(x)\})=0$.
	We will not distinguish between the function $f$ and its equivalence class, and we will say that $f = g$ whenever the functions $f$ and $g$ are equal almost everywhere.
	
	For all $1 \leq p < \infty$, $L^p(\Omega)$ is the set of equivalence classes of measurable functions $f: \Omega \rightarrow \R$ that satisfy
	$$
	\int_{\Omega} |f|^p dx < \infty.
	$$
	If $f \in L^p(\Omega)$, the $L^p$ norm of $f$ is defined by $$\norm{f}_p^p = \int_{\Omega} |f|^p dx.$$
	
	$L^\infty(\Omega)$ is the set of equivalence classes of measurable functions that are essentially bounded: we will say that $f: \Omega \rightarrow \R$ belongs to $L^\infty(\Omega)$ if there exists $y \in \R$ such that $\leb(\{x\in\Omega:|f(x)| > y\}) = 0$.
	In this case,
	$$\norm{f}_\infty = \inf\{y \in \R : \leb(\{x\in\Omega:f(x) > y\}) = 0 \}.$$
	\item The space $H^p(\Omega)$ is defined as
	$$H^p(\Omega) = \{ f \in L^2(\Omega) : D^\alpha f \in L^2(\Omega) \text{ for every multi-index } \alpha \text{ with }|\alpha|\leq p \}.$$
	We also consider the following norm over the space $H^p(\Omega)$:
	$$
	\norm{f}_{H^p} = \sum_{|\alpha|\leq p} \norm{D^\alpha f}_p,
	$$
	and we will call it the $H^p$ norm.
	For the properties of the spaces $H^p(\Omega)$, we refer to \cite{strichartz, tartar}.
	\item If $V(\Omega)$ is one of the above vector spaces, we define
	$$V(\Omega,\R^k) = \{ f: \Omega \rightarrow \R^k : f_i \in V(\Omega) \text{ for } i=1, \ldots,k \},$$
	where $f_i:\Omega\rightarrow\R$ is the $i$-th component of $f$.
	\item $\rad(\R) = \{ \nu : \nu \text{ is a Radon measure over } \R \text{ satisfying } |\nu|(\R)<+\infty \}$.
	\item $\prob(\R) = \{ \nu \in \rad(\R) : \nu \text{ is a probability measure}\}$.
\end{itemize}

Following \cite{balder, ball, webbym} and others, measurable functions $\nu : \Omega \rightarrow \prob(\R)$ will be called Young measures.
Measurable functions $\nu : \Omega \rightarrow \rad(\R)$ will be called parametrized measures, even though in the literature the term parametrized measure is used as a synonym for Young measure.
If $\nu$ is a parametrized measure and if $x \in \Omega$, we will write $\nu_x$ instead of $\nu(x)$.

\section{The measure-valued solutions to problem \ref{mark}}\label{section meassol}

In this section, we recall the notions of measure-valued solutions for problem \ref{mark} under the hypothesis that $u^+ < +\infty$ and under the hypothesis that $u^+ = +\infty$.
The most common approach to the study of problem \ref{mark} is to treat it as the limit of a sequence of approximating problems.

\subsection{The Young measure solution in the case $u^+ < +\infty$}

In \cite{plotnikov}, Plotnikov studied problem \ref{mark} % and \ref{mark II}
by means of the following Sobolev regularization:
\begin{equation}\label{sobolev regularization}
\left\{
\begin{array}{l}
u_t = \Delta v \text{ in } \Omega\\
v = \phi(u)+\eta u_t\\
\frac{\partial v}{\partial \normal} = 0 \text{ in } [0,T] \times \partial\Omega\\
u(0,x) = u_0 \text{ in } \Omega.
\end{array}\right.
\end{equation}
%\color{red}
%This approximation can be theoretically justified by arguing that equation \ref{mark} is a simplified mathematical model that does not take into account some relevant features of the original physical phenomena.
%The derivation of the above Sobolev regularization can then be justified by \ldots
%\textbf{Da completare. Menziona anche Cahn-Hilliard?}
%\color{black}
with $\eta > 0$.
The Neumann initial-boundary value problem for this regularized problem under the hypothesis that $u^+ < +\infty$ has been studied by Novick-Cohen and Pego \cite{stable patterns}. % and by Padron \cite{sobolev padron} under the hypothesis that $u^+ = + \infty$.
%
%For further details the reader is referred to Theorem 2.1 and to Proposition 2.4 of Novick-Cohen and Pego \cite{stable patterns}, to Theorem 1 of Padron \cite{sobolev padron} and to Theorem 2 of Mascia, Terracina and Tesei \cite{matete}.
%
In particular, Novick-Cohen and Pego proved that if $\phi$ is locally Lipschitz continuous and the initial data is $L^\infty(\Omega)$, then there exists a unique classical solution $(u_\eta,v_\eta)$ to problem \ref{sobolev regularization}, with $u_\eta \in C^1([0,T], L^\infty(\Omega))$ and $v_\eta = \phi(u_\eta)+\eta (u_\eta)_t$,
and the functions $(u_\eta, v_\eta)$ satisfy the inequality
\begin{equation}\label{entropy inequality classica}
\int_0^T \int_\Omega G(u_\eta) \varphi_t - g(v_\eta)\stgrad \varphi-g'(v_\eta)|\stgrad v_\eta|^2\varphi dx dt \geq 0
\end{equation}
for all non-decreasing $g \in C^1(\R)$ satisfying $G' = g$, and for all $\varphi \in \tests([0,T] \times \Omega)$ with $\varphi(t,x) \geq 0$ for all $(t,x) \in [0,T] \times \Omega$.
This inequality has the role of an entropy condition for the weak solutions of problem \ref{mark}, in a sense made precise by Evans \cite{evans survey} and by Mascia, Terracina and Tesei \cite{matete}.

Since the sequences $\{u_\eta\}_{\eta > 0}$ and $\{v_\eta\}_{\eta > 0}$ are uniformly bounded in $L^\infty([0,T] \times \Omega)$, they have a weak-$\star$ limit $(u,v)\in L^\infty([0,T] \times \Omega)$ that satisfies equation
\begin{equation}\label{pb limite}
\left\{
\begin{array}{l}
u_t = \Delta v \text{ in } \Omega\\
\frac{\partial v}{\partial \normal} = 0 \text{ in } [0,T] \times \partial\Omega\\
u(x) = u_0 \text{ in } \Omega
\end{array}\right.
\end{equation}
in the weak sense, i.e.\ $u \in L^\infty([0,T) \times \Omega)$, $v \in L^\infty([0,T) \times \Omega) \cap L^2([0,T],H^1(\Omega))$ such that
\begin{equation}\label{weak solution illposed}
\int_0^T \int_\Omega u \varphi_t - \nabla v \cdot \nabla \varphi dx dt + \int_{\Omega} u_0(x)\varphi(0,x) dx = 0
\end{equation}
for all $\varphi \in C^1([0,T] \times \overline{\Omega})$ with $\varphi(T,x) = 0$ for all $x \in \Omega$.
However, since in general weak-$\star$ convergence is not preserved by composition with a nonlinear function, we have no reason to expect that $v = \phi(u)$, so that the weak solution of \ref{pb limite} is not a weak solution of \ref{mark}.

Thanks to the uniform bound on the $L^\infty$ norm of $\{u_\eta\}_{\eta > 0}$, Plotnikov in \cite{plotnikov} showed that the sequence $\{u_\eta\}_{\eta > 0}$ has a limit point $\nu$ in the space of Young measures, and $\nu$ can be interpreted as a weak solution to equation \ref{mark}.
Moreover, $\nu$ can be characterized as a superposition of at most three Dirac measures concentrated at the three branches of $\phi$, and the Young measure $\nu$ and the function $v$ defined by $v = \int_{\R} \phi(\tau) d\nu$ still satisfy the entropy inequality \ref{entropy inequality classica}.
This analysis suggests the following definition of weak solution to problem \ref{mark} in the sense of Young measures.

\begin{definition}
	An entropy Young measure solution of problem \ref{mark} consists of funtions $u, v, \lambda_i \in L^\infty([0,T]\times\Omega)$, $1 \leq i \leq 3$, satisfying the conditions:
	\begin{enumerate}
		\item $\lambda_i \geq 0$, $\sum_{i = 1}^3 \lambda_i = 1$, with $\lambda_1(x) = 1$ if $v(x)< \phi(u^+)$, and with $\lambda_3(x) = 1$ if $v(x) > \phi(u^-)$;
		\item $v \in L^2([0,T], H^1(\Omega))$ and $u = \sum_{i = 1}^3 \lambda_i S_i(v)$, where $S_i(v)$ are defined as follows:
		$$
		\begin{array}{l}
		S_1 : (-\infty,\phi(u^-)] \rightarrow (-\infty,u^-],\\
		S_2 : (\phi(u^+),\phi(u^-)) \rightarrow (u^-,u^+),\\
		S_3 : [\phi(u^+),+\infty) \rightarrow [u^+,+\infty),\\
		\end{array}
		$$
		and, for all $i$, $u = S_i(v) \text{ iff } v = \phi(u)$;
		%\item ;
		\item $u_t = \Delta v$ in the weak sense, i.e.\
		$$
		\int_0^T \int_\Omega u \varphi - \nabla v \cdot \nabla \varphi dx dt + \int_{\Omega} u_0(x)\varphi(0,x) dx = 0
		$$
		for all $\varphi \in C^1([0,T] \times \overline{\Omega})$ with $\varphi(T,x) = 0$ for all $x \in \Omega$.
		\item for all $g \in C^1(\R)$ with $g' \geq 0$, define
		$$
		G(x) = \int_0^x g(\phi(\tau)) d\tau \text{ and } G^\star(u)= \sum_{i = 1}^3 \lambda_i G(S_i(v)).
		$$
		Then the following entropy inequality holds:
		\begin{equation}\label{entropy inequality}
		\int_0^T\int_{\Omega} G^\star(u) \varphi_t - g(v)\nabla \cdot v \nabla \varphi -g'(v)|\nabla v|^2\varphi dx dt \geq 0
		\end{equation}
		for all $\varphi \in \tests([0,T] \times \Omega)$ with $\varphi(t,x) \geq 0$ for all $(t,x) \in [0,T] \times \Omega$.
	\end{enumerate}
\end{definition}

It has been proved \cite{matete, plotnikov} that problem \ref{mark} allows for an entropy Young measure solution, but in general these solutions are not unique, as shown for instance in \cite{non-uniqueness}.
Uniqueness of Young measure solutions has been proved by Mascia, Terracina and Tesei \cite{matete2} under the additional constraint that the initial data and the solution do not take value in the interval $(u^-,u^+)$.
For a detailed discussion of the Young measure solutions to problem \ref{mark}, we refer to \cite{irreversibility, matete, matete2, plotnikov, non-uniqueness}.

\subsection{The Radon measure solution in the case $u^+ = +\infty$}

The Neumann initial-boundary value problem for \ref{sobolev regularization} under the hypothesis that $u^+ = +\infty$ has been studied by Padron \cite{sobolev padron}.
In analogy to the case where $u^+<+\infty$, if $\phi$ is Lipschitz continuous and the initial data is $L^\infty(\Omega)$, then there exists a unique classical solution $(u_\eta,v_\eta)$ to \ref{sobolev regularization}, with $u \in C^1([0,T], L^\infty(\Omega))$ and $v_\eta = \phi(u)+\eta u_t$.
However, while the sequence $\{v_\eta\}_{\eta > 0}$ is still uniformly bounded in the $L^\infty$ norm, the sequence $\{u_\eta\}_{\eta>0}$ is not, %uniformly bounded in $L^\infty([0,T] \times \Omega)$,
so we cannot take the weak-$\star$ limit of $\{u_\eta\}_{\eta>0}$, even in the sense of Young measures.
Nevertheless, thanks to the Neumann boundary conditions, $\norm{u_\eta(t)}_1 = \norm{u_0}_1$ for all $t \geq 0$.
As a consequence, the sequence $\{u_\eta\}_{\eta > 0}$ has a limit point $u$ in the space of positive Radon measures over $[0,T]\times\Omega$.
In \cite{smarrazzo}, it is proved that
%convergence for Young measures (see \cite{xxx,xxx}),
$u$ can be represented as the sum $u = u_r+\mu$, where $u_r$ is the baricenter of the Young measure $\nu(t,x)$ associated to an equi-integrable subsequence of $\{u_\eta\}_{\eta > 0}$, and $\mu$ is a Radon measure over $[0,T] \times \Omega$.
The function $u_r \in L^1([0,T]\times\Omega)$ is called the regular part of the solution $u$, while the Radon measure $\mu$ is called the singular part of $u$.
If we define $v$ as the $L^\infty$ function obtained as the weak-$\star$ limit of the sequence $\{v_\eta\}_{\eta > 0}$, then $u_r$, $\mu$ and $v$ are a weak solution to problem \ref{mark} in the sense that they satisfy the equation
\begin{equation}\label{soluzione smarrazzo}
\int_{0}^T \ldual \mu, \varphi_t\rdual dt + \int_0^T \int_{\Omega} \overline{u} \varphi_t - \nabla v \cdot \nabla \varphi dx dt + \int_{\Omega} u_0(x) \varphi(0,x) dx = 0,
\end{equation}
for any $\varphi \in C^1([0,T] \times \overline{\Omega})$ with $\varphi(T,x) = 0$ for all $x\in\Omega$.
In particular, a notion of entropy Radon measure solution can be defined for equation \ref{mark} even in the case $u^+=+\infty$.

\begin{definition}
	An entropy Radon measure solution of problem \ref{mark} consists of funtions $u_r, v, \lambda_i \in L^\infty([0,+\infty)\times\Omega)$, $i= 1,2$ and of a positive Radon measure $\mu \in \rad([0,T]\times\Omega)$, satisfying the conditions:
	\begin{enumerate}
		\item $\lambda_i \geq 0$, $\sum_{i = 1}^2 \lambda_i = 1$;
		\item $v \in L^2([0,+\infty), H^1(\Omega))$ and
		$$u_r = \left\{ \begin{array}{ll} \sum_{i = 1}^2 \lambda_i S_i(v) & \text{if } v(x) > 0\\
		0 & \text{if } v(x) = 0; \end{array} \right.$$
		%\item ;
		\item $(u_r+\mu)_t = \Delta v$ in the the sense of equation \ref{soluzione smarrazzo};
		\item the entropy inequality \ref{entropy inequality} holds for $u_r$ and $v$.
	\end{enumerate}
\end{definition}

Smarrazzo proved in \cite{smarrazzo} that problem \ref{mark} allows for a global entropy Radon measure solution, and we refer to her paper for an in-depth analysis of the properties of such solutions.

We conclude the discussion of the entropy Radon measure solution to problem \ref{mark} by recalling two features of the singular part of the solution.
In \cite{smarrazzo}, Smarrazzo showed that the singular part $\mu$ of the entropy Radon measure solution satisfies the following equality for all $t \geq 0$:
\begin{equation}\label{singular disintegration}
\mu(t) = \left( \int_{\Omega} u_0(x) dx - \int_{\Omega} u_r(t,x) dx \right)\tilde{\mu}(t),
\end{equation}
where $\tilde{\mu}(t)$ is a positive probability measure over $\Omega$.
Moreover, she conjectured that this singular term prevails over the regular term for large times.
In section \ref{coherence}, we will show that the solution to problem \ref{mark} obtained from the grid function formulation satisfies an equality analogous to \ref{singular disintegration}, and in section \ref{properties} we will show that the conjecture by Smarrazzo holds for such solutions.

\section{Grid functions and their properties}\label{section gridfz}

In this section, we recall some properties of grid functions studied in \cite{ema2} that will be useful for the nonstandard formulation of the ill-posed problem \ref{mark}.
For an in-depth discussion of the theory of grid functions, we refer to \cite{ema2}.
We will assume the reader to be familiar with the basics
of nonstandard analysis; a classic reference is Davis \cite{da} or Goldblatt \cite{go}.

Grid functions over $\Omega$ are functions of nonstandard analysis defined over a hyperfinite domain that coherently generalizes the open domain $\Omega$.
In the next definition, we introduce this hyperfinite domain.

\begin{definition}[The hyperfinite domain $\dom$]
	Let $N_0\in\ns{\N}$ be an infinite hypernatural number.
	Set $N = N_0!$ and $\varepsilon = 1/N$, and define %$\Lambda^k = \{ (n_1\varepsilon, \ldots, n_k\varepsilon) : n_1, \ldots, n_k \in [-N^2, N^2] \subseteq\,\! ^\ast\Z\}$ for all $k \in \N$.
	$$\Lambda = \{ n\varepsilon : n \in  [-N^2, N^2] \cap \ns{\Z}\}.$$
	%\item We set $\dt = \varepsilon^2/\alpha$ for some finite $\alpha \in \ns{\R}$ and we define $T = \{n\dt : n \in \ns{\Z} \mathrm{\ and\ } 0 \leq n \}$.
	%\item Against the general rule introduced above, we will denote a generic point $x \in \ns{\R}^{N+1}$ by $x = (x_0, x_1, \ldots, x_N)$.
	The hyperfinite domain is defined as $\dom = \ns{\Omega}\cap \Lambda^k$.
	Notice that $\dom$ is an internal subset of $\Lambda^k$, and in particular it is hyperfinite.
	
	We define also the $\Lambda$-boundary of $\dom$ as
	$$
	\partial_\Lambda \dom = \{ x \in \dom : \exists y \in \ns{A}^c \text{ satisfying } |x-y|_\infty\leq \varepsilon \}.
	$$
	
	We will say that $x \in \dom$ is nearstandard in $\Omega$ iff there exists $y \in \Omega$ such that $x \sim y$.
\end{definition}

In [6] it is proved that for every open set $\Omega$, $\sh{\dom} = \overline{\Omega}$ and $\sh{ (\partial_\Lambda \dom)} = \partial \Omega$.

Grid functions over $\Omega$ are hyperreal valued functions defined over $\dom$.

\begin{definition}[Grid functions over $\Omega$]\label{def grid functions}
	We will say that an internal function $f : \dom \rightarrow \hR$ is a grid function, and we denote by $\grid{\Omega}$ the space of grid functions defined over $\dom$:
	$\grid{\Omega} = \mathbf{Intl}(\hR^{\dom}) = \{ f : \dom \rightarrow \hR \text{ and } f \text{ is internal}\}.$
\end{definition}

Since grid functions are defined on a discrete domain, there is no notion of derivative for grid functions.
However, in nonstandard analysis it is fairly usual to replace the derivative by a finite difference operator with an infinitesimal step.

\begin{definition}[Grid derivative]\label{fd}
	For any $f \in \grid{\Omega}$, we define the $i$-th forward finite difference of step $\varepsilon$ as
	\begin{equation*}
	\D_i f(x)= \D_i^+ f(x)= \frac{f(x+\varepsilon e_i)-f(x)}{\varepsilon }
	\end{equation*}
	and the $i$-th backward finite difference of step $\varepsilon$ as
	$$
	\D_i^- f(x) = \frac{f(x)-f(x-\varepsilon e_i)}{\varepsilon }.
	$$
	%For $f : \Lambda^k \rightarrow \hR$, the functions $\D_{i} f(x)$ with $0 \leq i \leq k$ are defined in the expected way.
	If $n \in \ns{\N}$, $\D^n_i$ is recursively defined as $\D_i(\D_i^{n-1})$ and, if $\alpha$ is a multi-index, then $\D^\alpha$ is defined as expected:
	$$
	\D^\alpha f = \D_1^{\alpha_1} \D_2^{\alpha_2} \ldots \D_n^{\alpha_n} f.
	$$
	These definitions can be extended to $\D^-$ by replacing every occurrence of $\D$ with $\D^-$.
	%Most of the time, we will write $\D_\alpha$ for $\D_\alpha^+$.
\end{definition}

In the same spirit, integrals are replaced by hyperfinite sums.

\begin{definition}[Grid integral and inner product]\label{def inner}
	Let $f,g : \ns{\Omega} \rightarrow \hR$ and let $A \subseteq \dom \subseteq \Lambda^k$ be an internal set.
	We define
	$$
	\int_{A} f(x) d\Lambda^k = \varepsilon^k \cdot \sum_{x \in A} f(x)
	$$
	and
	$$
	%\begin{array}{rcl}
	\langle f, g \rangle
	=  \displaystyle \int_{\Omega_\Lambda} f(x) g(x) d\Lambda^k %\\ \\
	=  \displaystyle \varepsilon^k \cdot \sum_{\Omega_\Lambda} f(x)g(x).
	%\end{array}
	$$
	%with the convention that, if $x \not \in \ns{\Omega}$, $f(x) = g(x) = 0$.
\end{definition}

For further details about the properties of the grid derivative and the grid integral, we refer to \cite{imme2, imme1, ema2, watt, keisler, nsa working math}.

We will now define a grid function counterpart of the space of test functions.

\begin{definition}
	We say that a function $f \in \grid{\Omega}$  is of class $S^0(\Omega)$ iff
	$f(x)$ is finite for some nearstandard $x \in \dom$ and
	for every nearstandard $x, y \in \dom$, $x \sim y$ implies $f(x) \sim f(y)$.
	We say that $f$ is of class $S^l(\Omega)$ if $D^\alpha f \in S^0(\Omega)$ for every multi-index $\alpha$ with $\max \alpha \leq l$, and we say that $f$ is of class $S^{\infty}(\Omega)$ if $\D^\alpha f \in S^0(\Omega)$ for any standard multi-index $\alpha$.

	%Let $I$ be a near-interval.
	We define the algebra of grid test functions as follows:
	$$
	\test(\Omega) =
	\left\{ f \in S^{\infty}(\Omega) : \sh{\supp f} \subset \subset \Omega \right\}.
	$$
\end{definition}
	
	It is well-known that, if $\varphi \in S^k(\Omega)$, then $\sh{\varphi}\in C^k(\Omega)$ \cite{imme2,watt}.
	In Lemma 3.2 of \cite{ema2}, it is proved that the algebra of test function is the grid function counterpart of the space of standard test functions $\tests(\Omega)$, in the sense that if $\varphi \in \test(\Omega)$, then $\sh{\varphi} \in \tests(\Omega)$, and if $\varphi \in \tests(\Omega)$, then the restriction of $\ns{\varphi}$ to $\dom$ belongs to $\test(\Omega)$.
	
	We now introduce an equivalence relation based on the duality with grid test functions.
	
\begin{definition}
	Let $f, g \in \grid{\Omega}$.
	We say that $f \equiv g$ iff for all $\varphi \in \test(\Omega)$ it holds
	$
	\langle f, \varphi \rangle \sim \langle g, \varphi \rangle.
	$
	%We define the space of generalized distributions on $\dom$ as the quotient $\grid{\Omega}/\equiv$.
	We will call $\pi$ the projection from $\grid{\Omega}$ to the quotient
	$
	\grid{\Omega} / \equiv,
	$
	and we will denote by $[f]$ the equivalence class of $f$ with respect to $\equiv$.
\end{definition}

In \cite{ema2} it is proved that the space of grid functions generalizes the space of distributions, and that the finite difference operation generalizes the distributional derivative to the space of grid functions.
In particular, there exists a subspace of $\grid{\dom}/\equiv$ that is a real vector space isomorphic to the space of distributions, and the finite difference operators induce the distributional derivative on the quotient.

\begin{theorem}\label{mainthm}
	Let $\B = \left\{ f\in \grid{\Omega}\ |\ \langle f, \varphi \rangle \text{ is finite for all } \varphi\in\test(\Omega)\right\}.$
	The function $\Phi:(\B/\equiv) \rightarrow \tests'(\Omega)$ defined by
	$$
	\ldual \Phi([f]), \varphi \rdual
	=
	\sh{\langle f, \ns{\varphi} \rangle}
	$$
	is an isomorphism of real vector spaces. % between $\B/\equiv$ and $\tests'(\Omega)$.
	Under the additional hypothesis $f \in S^0(\dom)$, then $[f] = \sh{f}$.
	Moreover, the diagrams
	$$
	\begin{array}{rcl}
	\begin{array}{ccc}
	\B & \stackrel{\D^+}{\longrightarrow} & \B \\
	\Phi \circ \pi \downarrow & & \downarrow \Phi \circ \pi\\
	\mbox{}\tests'({\Omega})  & \stackrel{D}{\longrightarrow} & \tests'({\Omega})  \\
	\end{array}
	&
	and
	&
	\begin{array}{ccc}
	\B & \stackrel{\D^-}{\longrightarrow} & \B \\
	\Phi \circ \pi \downarrow & & \downarrow \Phi \circ \pi\\
	\mbox{}\tests'({\Omega})  & \stackrel{D}{\longrightarrow} & \tests'({\Omega})  \\
	\end{array}
	\end{array}
	$$
	commute.
\end{theorem}
\begin{proof}
	See Theorem 3.10, Corollary 3.11 and Theorem 3.16 of \cite{ema2}.
\end{proof}

Thanks to this result, if $f \in \B$, we can identify the equivalence class $[f]$ with the distribution $\Phi([f])$.

By composing finite difference operators, we obtain the grid function counterpart of many differential operators, such as the grid gradient, the grid divergence and the grid Laplacian.

\begin{definition}[Grid gradient and grid divergence]
	If $f\in\grid{\Omega}$, we define the forward and backward grid gradient of $f$ as:
	$$
	\grad^\pm f = (\D^\pm_{1}f, \ldots, \D^\pm_i, \ldots, \D^\pm_{k}f).
	$$
	In a similar way, if $f : \dom \rightarrow \hR^k$, we define the forward and backward grid divergence as
	$$
	\div^\pm f = \sum_{i = 1}^k \D^\pm_{i} f_i.
	$$
\end{definition}

In the sequel, we will mostly drop the symbol $+$ from the above definitions: for instance, we will write $\grad$ instead of $\grad^+$.

It is a consequence of Theorem \ref{mainthm} that, if $f\in S^{1}(\Omega)$, then $\sh{(\grad f)}$ is the usual gradient of $\sh{f}$, and similar results holds for $\grad^-$, $\div$, $\div^-$ and $\lap$.
Moreover, by Theorem \ref{mainthm}, the operators $\grad$ and $\grad^-$ satisfy the formula
$$
\sh{\langle \grad f, \ns{\varphi} \rangle} = \sh{\langle \grad^- f, \ns{\varphi} \rangle} = -\ldual [u], \stdiv\, \varphi \rdual
$$
for all $f \in \test'(\Omega)$ and for all functions $\varphi \in \tests(\Omega,\R^k)$,
and $\div$ and $\div^-$ satisfy the formula
$$
\sh{\langle \div f, \ns{\varphi} \rangle} = \sh{\langle \div^- f, \ns{\varphi} \rangle} = - \ldual [f], \nabla \varphi \rdual
$$
for all $f\in\test'(\Omega,\hR^k)$ and for all $\varphi \in \tests(\Omega)$.

In the sequel of the paper, we will need to discuss %find it useful to study the properties of
grid functions that are infinitesimally close to a $H^1$ function. 
We begin by introducing a notion of nearstandardness in the same spirit as the one that is routinely used for points in the euclidean space.
%REFER TO SOMEONE.
In order to introduce the notion of nearstandardness in $H^1$, we need to define the $L^p$ norms over the space of grid functions, and a canonical extension of $L^2$ functions to the space of grid functions.

\begin{definition}[$L^p$ norms for grid functions]
	For all $f \in \grid{\Omega}$, define
	$$
	\norm{f}_p^p = \varepsilon^k \sum_{x \in \dom} |f(x)|^p \text{ if } 1 \leq p < \infty, \text{ and } \norm{f}_\infty = \max_{x \in \dom} |f(x)|.
	$$
	Moreover, $\norm{f}_{H^1} = \norm{f}_2 + \norm{\grad^+ f}_2$.
\end{definition}

\begin{lemma}\label{lemma 4.2}
	Let $f, g \in \grid{\dom}$.
	If $\norm{f-g}_p \sim 0$ for some $1 \leq p \leq \infty$, then $f \equiv g$.
\end{lemma}
\begin{proof}
	See Corollary 4.2 of \cite{ema2}.
\end{proof}

\begin{definition}[Canonical extension from $L^2(\Omega)$ to $\grid{\Omega}$]
	If $f \in \grid{\Omega}$, let $\lp{f} \in \ns{L}^2(\R^k)$ defined by 
	$$
	\lp{f}(x) = \left\{
	\begin{array}{ll}
	f((n_1,n_2,\ldots,n_k)\varepsilon) & \text{if } n_i \varepsilon \leq x_i < (n_i+1)\varepsilon \text{ for all } 1 \leq i \leq k\\
	0 & \text{if } |x_i| > N \text{ for some } 	1 \leq i \leq k,
	\end{array}
	\right.
	$$
	with the agreement that $f((n_1,n_2,\ldots,n_k)\varepsilon) = 0$ if $(n_1,n_2,\ldots,n_k)\varepsilon \not \in \dom$.
	We define the canonical extension $P : L^2(\Omega) \rightarrow \grid{\Omega}$ as the $L^2$ projection over $\grid{\Omega}$, i.e.\ for all $g\in {L}^2(\Omega)$, $P(g)$ is the unique grid function that satisfies the equality
	$$
	\langle P(g), f \rangle = \ns{\int}_{\hR^k} \ns{g}(x) \lp{f}(x) dx
	$$
	for all $f \in \grid{\Omega}$.
\end{definition}

The properties of the canonical extension are recalled in the next Lemma.

\begin{lemma}\label{questo corollario}
	If $\Omega$ is bounded in $\R^k$, then for all $f \in L^2(\Omega)$
	\begin{itemize}
		\item $[P(f)] = f$;
		\item $\norm{\ns{f}-P(f)}_2 \sim 0$.
	\end{itemize}
\end{lemma}
\begin{proof}
	See Lemma 4.6 and Lemma 4.7 of [6].
\end{proof}

\begin{definition}[Nearstandardness in $H^1(\Omega)$]
	We will say that $f\in\grid{\Omega}$ is nearstandard in $H^1(\Omega)$ iff there exist $\overline{f}\in H^1(\Omega)$ that satisfies $\norm{f-P(\overline{f})}_{H^1} \sim 0$.
\end{definition}

As a consequence of Lemma \ref{lemma 4.2} and of Lemma \ref{questo corollario}, if $f$ is nearstandard in $H^1(\Omega)$ and $\overline{f}$ satisfies $\norm{f-P(\overline{f})}_{H^1} \sim 0$, then $[f] = \overline{f}\in H^1(\Omega)$ and $\norm{f}_{H^1}\in \fin$.

In the grid function formulation of the ill-posed problem \ref{mark}, we will need to determine the behaviour of $\grad^-\ns{\phi}(u)$.
In the next Lemma we will prove that, if $f$ is nearstandard in $H^1(\Omega)$ and if $g \in C^1(\R)$, then $\ns{g}(f)$ is nearstandard in $H^1(\Omega)$, and the grid derivative of the composite function behaves in the same way as the distributional derivative of $g([f])$.

\begin{lemma}\label{lemma derivata H^1}
If $f$ is nearstandard in $H^1(\Omega)$ and if $g \in C^1(\R)$, then $[\grad^-\ns{g}(f)] = [g'(f)]\nabla[f]$.
\end{lemma}
\begin{proof}
For all $1\leq i \leq k$, we have the equalities
\begin{eqnarray}
\notag
\D^-_i\ns{g}(f(x))
& = & \frac{\ns{g}(f(x)) - \ns{g}(f(x-\dx e_i))}{\varepsilon} \\
& = & \frac{\ns{g}(f(x))-g(f(x)-\varepsilon \D^+_i f(x-\dx e_i))}{\varepsilon \D^+_i f(x-\dx e_i)} \cdot \D^+_i f(x-\dx e_i).
\label{quiqui2}
\end{eqnarray}
Finiteness of $\norm{f}_{H^1}$ ensures that $\D^+_i f(x-\dx e_i) \in \fin$ for almost every $x \in \dom$, and that $\varepsilon \D^+_i f(x-\dx e_i) \sim 0$ for all $x \in \dom$.
As a consequence, $[\D^+_i f(x-\dx e_i)] = D_i[f](\sh{x})$.
%The finiteness of $\norm{u}_\infty$ implies that it also holds $[\D^-_i\ns{\phi}(u)] = [\phi'(u)]D_i[u]$.
By equation \ref{quiqui2}, we deduce also $[\D^-_i\ns{g}(f)] = [g'(f)]D_i[f]$.
\end{proof}
\color{black}

The next results summarize the relations between grid functions and parametrized measures, that are studied in depth in \cite{ema2}.

\begin{theorem}\label{parametrized measures}
	For every $f \in \grid{\Omega}$, there exists a parametrized measure $\nu : \Omega \rightarrow \rad(\R)$ such that, for all $g \in \bcf(\R)$ and for all $\varphi \in C^0_c(\Omega)$, it holds
	\begin{equation}\label{young equivalence equation}
	\sh{\langle \ns{g}(f), \ns{\varphi} \rangle}
	=
	\int_{\Omega} \left( \int_{\R} g d \nu_x \right) \varphi(x) dx.
	\end{equation}
	Moreover, for all $x \in \Omega$ and for all Borel $A \subseteq \R$, $0\leq \nu_x(A) \leq 1$.
	
	If $\nu$ is also Dirac, then $\nu_x$ is the Dirac measure centred at $[f](x)$ for a.e.\ $x \in \Omega$.
\end{theorem}
\begin{proof}
	See Theorem 4.13 and Theorem 4.14 of \cite{ema2}.
\end{proof}

We remark that the difference between $\nu_x(\R)$ and $1$ is due to $f$ being unlimited at some non-negligible fraction of the monad of $x$.% $\mu(x)\cap\Lambda^k$.

\begin{corollary}\label{corollario baricentro}
	For every $f \in \grid{\Omega}$, let $\nu : \Omega \rightarrow \rad(\R)$ the parametrized measure satisfying Theorem \ref{parametrized measures}, and let $f_r : \Omega \rightarrow \R$ be defined by
	$$
		f_r(x) = \int_{\R} d\nu_x.
	$$
	Then $f_r$ is a measurable function.
	Moreover, if $\norm{f}_1 \in \fin$, then $f_r \in L^1(\Omega)$.
\end{corollary}

\section{The grid function formulation for the ill-posed PDE} \label{ns models}

In this section, we will derive the grid function formulation for the ill-posed problem \ref{mark} by generalizing an elementary model for the diffusion equation developed by Hanqiao, St.\ Mary and Wattenberg \cite{watt}.
This approach will allow us to choose a suitable grid function counterpart to the operator $\Delta \phi(u)$.
Under the hypotheses \ref{ipotesi di partenza} over $\phi$, we will prove that the grid function formulation always has a unique well-defined solution.

\subsection{Derivation of the grid function formulation}
For a matter of commodity, in the derivation of the model we will use the image of a population that moves around the grid $\Lambda^k$ according to some basic rules.
The initial distribution of the population around the grid is described by an internal function $u_0 : \Lambda^k \rightarrow \ns{[0,1]}$ satisfying $\int_{\Lambda^k} u_0(\x)d\Lambda^k = c \in \ns{\R}$.
The value $u_0(\x)$ determines the number of individuals of the population inhabiting point $\x$ at time $t = 0$.

Let $\dt = \varepsilon^2$.
The population moves around the grid according to the following rules:
\begin{itemize}
	\item the $n$-th move occurs between time $(n-1)\dt$ and $n\dt$;
	\item at each jump the population at each grid point breaks into $(2k + 1)$ smaller groups:
	\begin{itemize}
		\item for $i = 1, \ldots, k$, a fraction $p_i(u((n-1)\dt,\x))$ of the population at $\x$ jumps to $\x + \dx \vec{e_i}$;
		\item for $i = 1, \ldots, k$, a fraction $p_i(u((n-1)\dt,\x))$ of the population at $\x$ jumps to $\x - \dx \vec{e_i}$;
		\item the remaining fraction $1 - 2\sum_{i = 1}^{k} p_i(u((n-1)\dt,\x))$ of the population at $\x$ remains at $\x$.
	\end{itemize}
\end{itemize}
In the above description, the functions $p_i$ are internal functions $p_i : \ns{\R} \rightarrow \ns{\R}$ satisfying
\begin{itemize}
	\item $0 \leq  p_i(r)$ for all $r \in \ns{\R}$;
	\item $\sum_{i = 1}^{k} p_i(r) \leq 1/2$ for all $r \in \ns{\R}$
\end{itemize}
for all $i = 1, \ldots, k$.
The properties of the functions $p_i$ determine the criteria used by the population to choose whether and how to jump to a nearby grid point.
In particular, in the model outlined above an individual chooses its next movement to move according only to local informations.
If the functions $p_i$ are constant and do not depend on $i$, then the above model coincides with the nonstandard model of diffusion discussed in \cite{watt}.
More complex behaviour can be described by different choices of functions $p_i$ and by introducing a spatial bias.

If we denote by $u(t,x)$ the population present at time $t$ at point $\x$, % and observe that it holds $u(\x,0) = v(\x)$.
then by arguing as in section III of \cite{watt} we deduce that $u(t,\x)$ evolves according to the finite difference initial value problem
\begin{eqnarray*}
	u(0,\x) & = & u_0(\x) \\
	u((n+1)\dt,\x) & = & \left( 1-2\sum_{i = 1}^{k} p_i(u(n\dt,\x)) \right)u(n\dt,\x) \\
	& & + \sum_{i = 1}^{k} p_i(u(n\dt,\x+\dx e_i))u(n\dt,\x+\dx e_i) \\
	& & + \sum_{i = 1}^{k} p_i(u(n\dt,\x-\dx e_i))u(n\dt,\x-\dx e_i))
\end{eqnarray*}
From the above equation,
if we define $\phi_i(u(n\dt,\x)) = p_i(u(n\dt,\x))u(n\dt,\x)$, we obtain
$$
u((n+1)\dt,\x) - u(n\dt,\x) = \sum_{i = 1}^{k} \bigg[\phi_i(u(n\dt,\x+\dx e_i))
-2 \phi_i(u(n\dt,\x))
+ \phi_i(u(n\dt,\x-\dx e_i))\bigg].
$$
At this point, we divide both sides of the above equation by $\dt = \dx^2$ and obtain
$$
\frac{u((n+1)\dt,\x) - u(n\dt,\x)}{\dt} = \sum_{i =1}^{k} \D_i^+ \D_i^- \phi_i(u(\x,t)).
$$
If $\phi_i = \phi$ for all $i = 1, \ldots, k$, i.e.\ if the population moves without spatial bias, from the above equality we deduce
\begin{equation}\label{finite difference mark}
\frac{u((n+1)\dt,\x) - u(n\dt,\x)}{\dt} = \lap \phi(u).
\end{equation}

Notice that the above equality only holds for all $x \not \in \partial_\Lambda\dom$, since if $x \in \partial_\Lambda\dom$ the term $\lap \ns{\phi(u(t,x))}$ depends upon points on the outside of $\dom$, sometimes called ghost points in the literature of numerical methods for the solution of differential equations.
Drawing inspiration from this field, at the points of $\partial_\Lambda\dom$ we approximate the Laplacian with Neumann boundary conditions with one-sided difference formulas that involve these ghost points.
More precisely, let
$$
I_x^+ = \{ i : x + \varepsilon e_i \not \in \dom \} \text{ and } I_x^- = \{ i : x - \varepsilon e_i \not \in \dom \}.
$$
Notice that $I_x^+ \cup I_x^- \ne \emptyset$ if and only if $x \in \partial_\Lambda \dom$; moreover, whenever $i \in I_x^+$ $\D^+_i \ns{\phi(u(x,t))}$ is the $i$-th grid derivative that points towards the exterior of $\dom$, and the same applies to $\D^-_i \ns{\phi(u(t,x))}$ whenever $i \in I_x^-$.
As a consequence, an approximation of the Neumann boudary conditions $\frac{\partial \phi(u)}{\normal}=0$ is obtained by imposing that these grid derivatives satisfy
\begin{equation}\label{neumann}
\sum_{i \in I_x^+} \D^+_i \ns{\phi(u(t,x))} = 0 \text{ and } \sum_{i \in I_x^-} \D^-_i \ns{\phi(u(t,x))}=0
\end{equation}
for all $x \in \dom$.
By substituting these equalities in the formula of $\lap \ns{\phi(u(t,x))}$, we obtain the following first-order discrete approximation of the Laplacian with Neumann boundary conditions:
\begin{eqnarray*}
	\lap \ns{\phi(u(t,x))}
	&=&
	-\varepsilon^{-1}\sum_{i \in I_x^+} \D^-_i \ns{\phi(u(t,x))}
	+
	\varepsilon^{-1}\sum_{i \in I_x^-} \D^+_i \ns{\phi(u(t,x))}+\\
	&&+
	\sum_{i \not \in I_x^+ \cup I_x^-}\D_i^+ \D_i^- \phi(u(t,x)).
\end{eqnarray*}
%This approximation is usually avoided in numerical computations in favour of higher-order schemes.
More sophisticated approximations of the normal derivative at boundary points of $\Omega$ can be obtained from different schemes; for an overview we refer to \cite{numerical1} and to \cite{numerical2}.

The above argument suggests that the operator  $\lap \ns{\phi}$ is a coherent generalization to the space of grid functions of the operator $F: L^\infty(\Omega)\cap H^1(\Omega) \rightarrow (C^1(\overline{\Omega}))'$ defined by
\begin{equation}\label{f}
\langle F(u), \varphi \rangle_{C^1(\Omega)} = - \int_{\Omega} \nabla \phi(u) \cdot \nabla \varphi dx
\end{equation}
for all $\varphi \in C^1(\overline{\Omega})$.
We will now prove that $\lap \ns{\phi}$ is indeed coherent with $F$ in the sense of Theorem 5.8 of \cite{ema2}.
In particular, we will prove that whenever $u$ is nearstandard in $H^1(\Omega)$, the following conditions are satisfied:
\begin{enumerate}
	\item $[\lap \ns{\phi}(u)]\in (C^1(\overline{\Omega}))'$;
	\item the above expression is well-defined in the sense that it does not depend upon infinitesimal perturbations over $u$;
	\item the functional $[\lap \ns{\phi}(u)]$ is equal to the functional $F([u])$;
	\item conversely, if $\overline{u} \in H^1(\Omega)$, then $[\lap \ns{\phi}(P(\overline{u}))]$ is equal to $F(\overline{u})$.
\end{enumerate}
Notice how condition (1) of Theorem 5.8 of \cite{ema2} is replaced by a different coherence condition that depends upon the definition of $F$.

\begin{proposition}\label{coerenza formulazione discreta}
	Let $\phi$ be a standard function satisfying hypotheses \ref{ipotesi di partenza}, and let $F$ be defined by equation \ref{f}.
	Suppose also that $u \in \grid{\Omega}$ is nearstandard in $H^1(\Omega)$ and that $\norm{u}_\infty \in \fin$.
	Then
	\begin{enumerate}
		\item $[\lap \ns{\phi}(u)] \in (C^1(\overline{\Omega}))'$;
		\item if $v \in \grid{\Omega}$ satisfies $\norm{v}_{H^1} \sim 0$, then $[\lap \ns{\phi}(u)]=[\lap \ns{\phi}(u+v)]$;
		\item $[\lap \ns{\phi}(u)] = F([u]) \in (C^1(\overline{\Omega}))'$;
		\item for all $\overline{u} \in L^\infty(\Omega)\cap H^1(\Omega)$,
		$[\lap \ns{\phi}(P(\ns{\overline{u}}))] = F(\overline{u})$.
	\end{enumerate}
\end{proposition}
\begin{proof}
	Point (1).
	By the discrete summation by parts formula and by taking into account the Neumann boundary conditions \ref{neumann}, for all $\varphi \in S^{1}(\overline{\Omega}_\Lambda)$ we have the equality
	%	\begin{eqnarray*}
	%	\langle \lap \ns{\phi}(u), \varphi \rangle &=& - \langle \grad^- \phi(u(x+\varepsilon)), \grad^+ \varphi \rangle +\\
	%	& & + \sum_{x \in \partial_\Lambda \dom} \left( \sum_{i \in I_x^+} \D^+_i\ns{\phi}(u(x)) \varphi(x+\dx e_i) + \sum_{i \in I_x^-} \D^-_i \ns{\phi}(u(x)) \varphi(x) \right).
	%	\end{eqnarray*}
	%Equation \ref{neumann} implies that the sum over $\partial_\Lambda \dom$ vanishes, so we have
	\begin{equation}\label{support}
	\langle \lap \ns{\phi}(u), \varphi \rangle = - \left\langle \grad^- \phi\left(u\left(x+\varepsilon\textstyle\sum_{i =1}^k e_i\right)\right), \grad^+ \varphi \right\rangle.
	\end{equation}
	Taking into account Theorem \ref{mainthm} and Lemma \ref{lemma derivata H^1}, we have
	\begin{eqnarray}
	\notag
	\sh{\langle} \lap \ns{\phi}(u), \varphi \rangle
	&=&
	\ldual [\lap \ns{\phi}(u)], \sh{\varphi} \rangle_{C^1(\Omega)} \\ \notag
	&=&
	- \int_\Omega [\nabla^- \ns{\phi}(u)] \cdot \nabla \sh{\varphi} dx \\
	\label{important}
	&=&
	-\int_\Omega \nabla[ \ns{\phi}(u)] \cdot \nabla \sh{\varphi} dx\\
	&=&
	-\int_\Omega [\phi'(u)]\nabla[u] \cdot \nabla \sh{\varphi} dx.
	\label{qui}
	\end{eqnarray}
	%and, by the hypothesis that $\D_i^+ u$ are nearstandard in $H^1(\Omega)$ for $1 \leq i \leq k$, $[u]\in H^1(\Omega)$.
	Since $u$ is nearstandard in $H^1(\Omega)$, $[u] \in H^1(\Omega)$.
	Moreover, the hypothesis $\norm{u}_\infty \in \fin$ ensures that $[u]\in L^\infty(\Omega)$.
	As a consequence, $[u]\in L^\infty(\Omega)\cap H^1(\Omega)$.
	This and hypotheses \ref{ipotesi di partenza} over $\phi$ entail that the integral \ref{qui} is finite, so that $[\lap \ns{\phi}(u)] \in (C^1(\overline{\Omega}))'$.
	
	Part (2).
	%We will now prove that $\norm{u-v}_\infty \sim 0$ and $\norm{u-v}_{H^1} \sim 0$ imply $[\lap \ns{\phi}(u)]=[\lap \ns{\phi}(v)]$.
	By S-continuity of $\ns{\phi}$ and of $\ns{\phi'}$, we have $$\norm{\ns{\phi}(u)-\ns{\phi}(u+v)}_\infty \sim \norm{\ns{\phi'}(u)-\ns{\phi'}(u+v)}_\infty \sim 0.$$
	By Lemma \ref{lemma 4.2}, $\norm{\ns{\phi'}(u)-\ns{\phi'}(u+v)}_\infty \sim 0$ implies $\ns{\phi'}(u)-\ns{\phi'}(u+v)\equiv 0$, that is equivalent to $\ns{\phi'}(u)\equiv\ns{\phi'}(u+v)$.
	We deduce that $[\ns{\phi'}(u)] = [\ns{\phi'}(u+v)]$.
	The assumption $\norm{v}_{H^1} \sim 0$ entails also $$\stgrad[u]=[\grad^+ u] = [\grad^+ (u+v)]=\stgrad[u+v],$$
	so that
	$$[\grad^+\ns{\phi}(u)] = [\phi'(u)]\nabla[u] = [\phi'(u+v)]\nabla[u+v] = [\grad^+\ns{\phi}(u+v)].$$
	Thanks to the previous equalities, from equation \ref{qui} we obtain
	$$
	\ldual [\lap \ns{\phi}(u)] - [\lap \ns{\phi}(u+v)], \sh{\varphi} \rangle_{C^1(\Omega)} =
	\int_\Omega ([\phi'(u)]\nabla^-[u] - [\phi'(u+v)]\nabla^-[u+v]) \cdot \nabla \sh{\varphi} dx = 0,
	$$
	so that the proof of point (2) is concluded.

	In order to prove point (3), notice that the hypotheses over $u$ entail that the Young measure associated to $u$ is Dirac for almost every $x \in \Omega$.
	As a consequence, by Theorem 3.13 of \cite{ema2} the Young measure associated to $u$ is a.e.\ equal to $[u]$.
	This implies $[\ns{\phi}(u)] = \phi([u])$, so that from equality \ref{important} we deduce that $[\lap \ns{\phi}(u)]=F([u])$, as desired.

	Part (4) of the assertion is a consequence of part (3), since if $\overline{u} \in L^\infty(\Omega)\cap H^1(\Omega)$, then Lemma \ref{lemma 4.2} and Lemma \ref{questo corollario} ensure that $P(\ns{\overline{u}})$ satisfies the hypotheses (3) and that $[P(\ns{\overline{u}})]=\overline{u}$.
\end{proof}

\subsection{The grid function formulation for the ill-posed PDE}
We now have all of the elements to formulate problem \ref{mark} in the sense of grid functions.

\begin{definition}
	The functions $[u], [\ns{\phi(u)}] : [0,+\infty) \rightarrow \grid{\Omega}/\equiv$ are called a grid solution of \ref{mark} if $u$
	%A grid function $u\in\ns{C}^1(I, \grid{\Omega_\Lambda})$
	%is called a Grid Solution (GS) of \ref{mark} if it
	satisfies the following system of ODEs:
	\begin{equation}\label{pb discreto}
	\left\{ \begin{array}{c}
	%u\in\mathcal{C}^{1}(I^{\ast},\Lambda(\Omega_\Lambda));\\
	u_t =\Delta_\Lambda\ns{\phi}(u);\\
	u\left(0,x\right)=\ns{u}_{0}\left(x\right).
	\end{array}\right.
	\end{equation}
\end{definition}

\begin{remark}
	A standard counterpart of system \ref{pb discreto} with $\dom = [0,1]\cap\Lambda$ and with standard $N$ has been used by Lizana and Padron \cite{sd} to describe the dynamics of a population inhabiting a finite collection of $N+1$ equally spaced points $\{0,\ldots,i/N,\ldots,1\}$ on the interval $[0,1]$.
	By the transfer principle, many properties of the finite model discussed in section 3 of \cite{sd} hold also for the hyperfinite system \ref{pb discreto}.
	Conversely, many of the results discussed in Sections \ref{properties} and \ref{sezione riemann} of this paper can be applied to this finite model by omitting the stars and by taking $N \in \N$.
\end{remark}

\begin{remark}
	We are interested not only in the solutions to problem \ref{pb discreto}, but also in the coherence with the measure-valued solutions to problem \ref{mark}.
	For this reason, we will restrict our attention to the study of problem \ref{pb discreto} under the hypotheses \ref{ipotesi di partenza} over $\phi$, and where the initial data is the nonstandard extension of a function $u_0 \in L^\infty(\Omega)$ that satisfies $u_0(x) \geq 0$ for all $x \in \Omega$.
	
	Notice however that problem \ref{pb discreto} makes sense for an arbitrary $f \in \ns{C^1(\R)}$ instead of $\ns{\phi}$ and for arbitrary initial data.
	This general form of problem \ref{pb discreto} will be addressed in a subsequent paper.
\end{remark}

Problem \ref{pb discreto} is a hyperfinite system of ordinary differential equations: as such, the existence of solutions and their properties can be studied by the theory of ordinary differential equations.
These results, in turn, apply to the grid solution for problem \ref{mark}.

\begin{theorem}\label{esistenza}
	There exists a maximal interval $I \subseteq \hR$ such that Problem \ref{pb discreto} has a unique solution $u \in \ns{C^1(I,\grid{\Omega})}$;
	Moreover, %$u(t) \in \simplex(u_0)$ for all $t \in I$.
	%In particular,
	$\norm{u(t)}_1 = \norm{u_0}_1$ for all $t \in I$.
\end{theorem}
\begin{proof}
	By transfer, existence and uniqueness can be deduced from the theory of ordinary differential equations.
	
	In order to prove that $\norm{u(t)}_1 = \norm{u_0}_1$ for all $t \in I$, notice that it holds
	$$
	\frac{d}{dt} \int_{\Omega_\Lambda} u(t,x) d\Lambda^k
	=
	\int_{\Omega_\Lambda} u_t(t,x) d\Lambda^k
	=
	\int_{\Omega_\Lambda} \Delta_\Lambda\phi(u(t,x)) d\Lambda^k.
	$$
	Thanks to the Neumann boundary conditions, $\int_{\Omega_\Lambda} \Delta_\Lambda\phi(u(t,x)) d\Lambda^k = 0$, so that the mass of the solution is preserved.
	%The hypotheses over $\phi$ ensure that at any point on the boundary of $\simplex_{\vsum{v}}$ the vector field \ref{pb discreto} points towards the interior of $\simplex_{\vsum{v}}$, and this completes the proof.
	%\color{red}
	%\textsl{(3)} is a consequence of \textsl{(2)} and of the boundedness of $\simplex_{\vsum{u_0}}$.
\end{proof}

\begin{proposition}[Invariant set]\label{ex invariant set}
	For all $t \in I$ and for all $x \in \dom$,
	\begin{enumerate}
		\item if $u^+ = +\infty$, then $u(t,x) \geq 0$.
		\item if $u^+ < +\infty$, then $u(t,x) \in [0,\max\{ \norm{\ns{u}_0}_\infty, S_3(\phi(u^-)) \} ]$.
		In particular, $\norm{u(t)}_\infty \in \fin$ is homogeneously bounded for all $t \geq 0$.
	\end{enumerate} 
	
\end{proposition}
\begin{proof}
	If $u^+ = +\infty$, let
	$$\overline{t} = \sup\{ \tau \geq 0 : u(t,x) \geq 0 \text{ for all } t \leq \tau \text{ and for all } x \in \dom \}.$$
	The hypotheses over $\phi$ and the definition of $\overline{t}$ ensure that %$\simplex_{\vsum{v}}$
	if $u(\overline{t},x) = 0$, then $u_t(\overline{t},x) = \Delta_\Lambda \phi(u(\overline{t},x))\geq 0$.
	As a consequence, $\overline{t} = \sup{I}$. 
	
	Similarly, if $u^+ < +\infty$, let
	$$\overline{t} = \sup\{ \tau \geq 0 : u(t,x) \in [0,\max\{ \norm{\ns{u}_0}_\infty, S_3(\phi(u^-)) \} ] \text{ for all } t \leq \tau \text{ and for all } x \in \dom \}.$$
	In this case, if $u(\overline{t},x) = 0$, the equality $u_t(\overline{t},x) = \Delta_\Lambda \phi(u(\overline{t},x)) \geq 0$ holds as in the previous part of the proof.
	If $u(\overline{t},x) = \max\{ \norm{\ns{u}_0}_\infty, \phi(u^-)\}$, %\ref{pb discreto}
	then a similar calculation allows to conclude $u_t(\overline{t},x) \leq 0$.
	We deduce that it holds $\overline{t} = \sup{I}$ also for this case. 
\end{proof}

Since for any initial data $u_0 \in L^\infty(\Omega)$ the invariant set for the dynamical system \ref{pb discreto} is bounded, we deduce global existence in time.

\begin{corollary}[Global existence in time]\label{global ex}
	The solution $u$ of system \ref{pb discreto} satisfies $u \in \ns{C^1}(\ns{[0,\infty)}, \grid{\Omega})$.
\end{corollary}
\begin{proof}
	Let $u$ be the solution of Problem \ref{pb discreto}, and let $I$ be the interval over which $u$ is defined.
	Define also
	$$\simplex^+(\ns{u_0})=\left\{f\in\grid{\Omega}:f(x) \geq0 \text{ for all } x\in\dom \text{ and } \norm{f}_1 = \norm{\ns{u}_0}_1\right\}.$$
	By Theorem \ref{esistenza} and by Proposition \ref{ex invariant set},
	$u(t) \in \simplex^+(\ns{u_0})$ for all $t \in I$.
	%and, if $u^+ < +\infty$,
	%$$u(t) \in \simplex(\ns{u_0})\cap\left\{f\in\grid{\Omega}:f(x) \in[0,\max\{ \norm{\ns{u}_0}_\infty, S_3(\phi(u^-))\}] \text{ for all } x\in\dom\right\}.$$
	
	Let $\dom = \{x_1, \ldots, x_M\}$, with $M = |\dom|$.
	We identify $u$ with a vector-valued function that, by abuse of notation, we will still denote by $u : I \rightarrow \hR^M$, with the convention that the $k$-th component of $u(t)$ is $u(t,x_k)$.
	Since the set $\simplex^+(\ns{u_0})$
	is $\ns{}$compact in $\hR^M$, the theory of ODEs allows to conclude that $u$ has global existence in time.
\end{proof}

As a consequence of Theorem \ref{esistenza} and of Corollary \ref{global ex}, we deduce existence, uniqueness, and global existence in time for the grid solution to problem \ref{mark}.
%With respect to uniqueness of the grid solution, (vedi Cutland) a possible source of non-uniqueness is the existence of

\begin{proposition}
	Problem \ref{mark} always has a unique global grid solution.
\end{proposition}

\section{Coherence of the grid solution with the measure-valued solutions to the ill-posed PDE}\label{coherence}

This section is devoted to the study of the coherence of the grid solution with the notions of measure-valued solutions for problem \ref{mark} discussed in section \ref{section illposed}.
In particular, we will show that, if $u$ is regular enough, then the grid solution of problem \ref{mark} coincides with an entropy Young measure solution in the case where $u^+ < +\infty$, and with the entropy Radon measure solution in the case where $u^+=+\infty$.

Our argument relies on an equality that will be used to establish an entropy condition for the pair $[u], [\ns{\phi(u)}]$.

\begin{lemma}\label{lemma divergenza}
	For all internal $f, g : \Lambda^k \rightarrow \ns{\R}$, it holds
	$$
	\div^-\big(g(f(x)) \cdot \grad f(x)\big)
	=
	g(f(x))  \lap(f(x))
	+ \grad^-f(x) \cdot \grad^- g(f(x)).
	$$
\end{lemma}
Notice that the above result is independent of the regularity of $f$ and $g$.

\begin{lemma}[Entropy condition]\label{lemma entropy condition}
	For any $g \in C^1(\R)$ with $g' \geq 0$, define
	$
	G(u(t,x)) = \int_0^{u(t,x)} g(\phi(s)) ds.
	$
	Then, if $u$ is the solution to problem \ref{pb discreto}, it holds
	$$
	\ns{G}(u)_t = \div^-((\ns{g}(\phi(u)) \grad^+(\phi(u))) - \grad^- \ns{g}'(\phi(u)) \cdot \grad^- \phi(u).
	$$
	and, if $\grad^- \phi(u)$ is finite,
	\begin{equation}\label{entropy condition}
	\ns{G}(u)_t \sim \div^-((\ns{g}(\phi(u)) \grad^+(\phi(u))) - \ns{g}'(\phi(u))  |\grad^- \phi(u)|^2.
	\end{equation}
\end{lemma}
\begin{proof}
	For $G$, $g$ and $u$ it holds
	$$
	G(u)_t  =  g(\phi(u))u_t 
	=  g(\phi(u)) \lap \phi(u).
	$$
	By Lemma \ref{lemma divergenza},
	$$
	\div^-((g(\phi(u)) \grad^+(\phi(u)))
	=
	g(\phi(u)) \lap \phi(u)
	+  \grad^- \ns{g}(\phi(u)) \cdot \grad^- \phi(u),
	$$
	so that the first equality is proved.
	
	For the second equality, notice that if $\grad^- \phi(u)$ is finite then, by Lemma \ref{lemma derivata H^1} and by the hypothesis $g\in C^1(\R)$, $\grad^- g(\phi(u)) \sim g'(\phi(u)) \grad^-\phi(u)$.
	As a consequence,
	$$
	\grad^- g(\phi(x)) \cdot \grad^-\phi(x)
	\sim
	g'(\phi(x)) |\grad^- \phi(x)|^2,
	$$
	as desired.
\end{proof}

Formula \ref{entropy condition} can be regarded as an entropy condition for system \ref{pb discreto}.
In particular, this equality allows us to prove that the solution obtained by the nonstandard model \ref{pb discreto} retains the physical meaning of an entropy solution.% for problem \ref{boh}.

Now, we will prove that the grid solution of problem \ref{mark} is always a very weak solution in the sense of distributions.

\begin{lemma}\label{lemma equivalenza 1}
	Let $u$ be the solution of problem \ref{pb discreto}.
	Then $[u]\in \tests'(\R\times\Omega)$, $[\ns{\phi}(u)] \in L^\infty(\R\times\Omega)$, and the couple $([u],[\ns{\phi}(u)])$ is a very weak solution of problem \ref{mark} in the sense of distributions, in the sense that $[u]$ and $[\ns{\phi}(u)]$ satisfy
	\begin{equation}\label{less regular}
	\int_0^T \langle [u], \varphi_t \rangle + \langle [\ns{\phi}(u)], \Delta \varphi \rangle dt + \int_{\Omega} u_0(x) \varphi(0,x) dx = 0
	\end{equation}
	for all $\varphi \in C^1([0,T], \tests'(\Omega))$ with $\varphi(T,x) = 0$ for all $x \in \Omega$.
\end{lemma}
\begin{proof}
	%Since $u$ is a solution to \ref{pb discreto}, for all $t \in [0,T]$ and for all $x \in \Omega_\Lambda$, $\partial_t u(t,x) = \Delta_{\Lambda}\phi(u(t,x))$.
	%Moreover, 
	By Proposition \ref{ex invariant set}, $\norm{u(t)}_1 \in \fin$ for all $t\in\hR_{\geq 0}$ and, by Proposition \ref{ex invariant set} if $u^+ < + \infty$ or by the boundedness of $\phi$ if $u^+ = +\infty$, also $\norm{\ns{\phi}(u)}_\infty \in \fin$ for all $t \in \hR_{\geq0}$.
	
	Now let $\varphi \in C^1([0,T], \tests'(\Omega))$ with $\varphi(T,x) = 0$ for all $x \in \Omega$, and define $\varphi_\Lambda(t) = \ns{\varphi(t)}_{|\Lambda}$.
	Since $u\in \ns{C}^{1}(\hR_{\geq0}, \grid{\Omega_\Lambda})$ and $\varphi_\Lambda \in \ns{C}^{1}([0,T], \test(\Omega_\Lambda))$, we have
	\begin{equation}\label{intermedio equivalenza 1}
	\int_0^T \langle u_t(t), \varphi_\Lambda(t) \rangle dt = - \int_0^T \langle u(t), (\varphi_\Lambda)_t(t) \rangle dt - \langle \ns{u}_0, \varphi_\Lambda(0,x) \rangle.
	\end{equation}
	By the discrete summation by parts formula, for all $t \in \hR_{\geq 0}$ we have
	\begin{equation}\label{intermedio equivalenza 2}
	\langle \Delta_{\Lambda}\phi(u(t)), \varphi_\Lambda(t) \rangle = \langle  \phi(u(t)), \Delta_\Lambda \varphi_\Lambda(t) \rangle.
	\end{equation}
	Taking into account that $u$ satisfies \ref{pb discreto}, by equations \ref{intermedio equivalenza 1} and \ref{intermedio equivalenza 2}, we obtain
	$$
	\int_0^T \langle u(t), (\varphi_\Lambda)_t(t) \rangle + \langle \phi(u(t)), \Delta_\Lambda \varphi_\Lambda(t) \rangle dt + \langle \ns{u}_0, \varphi_\Lambda(0,x) \rangle = 0.
	$$
	By Lemma \ref{questo corollario}, by $u_0\in L^\infty(\Omega)$ and by boundedness of $\Omega$,
	$$
	\sh{\langle \ns{u}_0, \varphi_\Lambda(0,x) \rangle} = \int_{\Omega} u_0 \varphi(0,x) dx.
	$$
	As a consequence, $[u]$ and $[\ns{\phi}(u)]$ satisfy
	$$
	\sh{\left(}\int_0^T \langle u(t), \ns{\varphi}(t) \rangle dt\right)
	=
	\int_{[0,T]} \ldual [u], \varphi_t \rdual dt
	+
	\int_{\Omega} u_0(x) \varphi(0,x) dx
	$$
	and
	$$
	\sh{\left(}\int_0^T \langle \Delta_{\Lambda}\phi(u(t)), \ns{\varphi}(t) \rangle dt\right) = \int_0^T \ldual [\ns{\phi}(u)], \Delta \varphi \rdual dt.
	$$
	By taking the sum of the previous equations, we deduce that equality \ref{less regular} holds.
\end{proof}

\subsection{The case $u^+ < +\infty$}

We will now discuss coherence of the grid solution with the solutions of problem \ref{mark} in the case where $u^+ < \infty$.
As expected, if $u$ is regular enough, then the grid solution to problem \ref{mark} is a solution of problem \ref{mark} in a classical sense.
The degree of regularity of the standard solution depends upon the regularity of $u$.

\begin{theorem}\label{mina}
	Let $[u], [\ns{\phi(u)}]$ be the
	grid solution of Problem \ref{pb discreto}, and let $\nu(t,x)$ the Young measure associated to $u$ according to Theorem \ref{parametrized measures}.
	\begin{enumerate}
		\item\label{ym} If
		%\begin{itemize}
		%\item there exists $U \in S^{0}(\hR_{\geq 0},\grid{\Omega})$ such that $U(t) \in [u(t)]$ for all $t \geq 0$,
		%\item
		$[\ns{\phi}(u)] \in L^2([0,T],H^1(\Omega))$,
		%\end{itemize}
		then $[u], [\ns{\phi(u)}]$ is an entropy Young measure solution of Problem \ref{mark} in the sense of equation \ref{weak solution illposed}.
		\item\label{vweak} If $\nu(t,x)$ is Dirac a.e.,%for a.e.\ $(t,x) \in [0,+\infty)\times\Omega$,
		then $[u]\in L^\infty([0,T], L^\infty(\Omega))$, $[\ns{\phi}(u)] = \phi([u])$, and $([u], [\ns{\phi}(u)])$ is a very weak solution of Problem \ref{mark}.
		\item\label{weak} Under the hypotheses
		\begin{itemize}
			%\item $u_0 \in \H^1(\Omega_\Lambda)$,
			\item $\nu(t,x)$ is Dirac a.e.,
			\item $[\ns{\phi}(u)] \in L^\infty([0,T] \times \Omega) \cap L^2([0,T],H^1(\Omega))$,
		\end{itemize}
		then $[u]$ is a weak solution of Problem \ref{mark}.
		\item\label{classic} If
		%\begin{itemize}
		%\item $\nu(t,x)$ is Dirac,
		%\item
		$u\in S^{1}(\ns{[0,+\infty)},S^{2}(\Omega))$,
		%\end{itemize}
		then $[u] = \sh{u}$ is a classical global solution of Problem \ref{mark}. 
	\end{enumerate}
\end{theorem}
\begin{proof}
	(\ref{ym}).
	Since $[\ns{\phi}(u(t))] \in H^1(\Omega)$ for a.e.\ $t \geq 0$, we deduce that $\int_{\Omega}\phi(\tau)d\nu(t,x)$ is single-valued for a.e.\ $t \geq 0$.
	In particular, $\nu(t,x)$ is a.e.\ a superposition of at most three Dirac measures centred at $S_i\left(\int_{\Omega}\phi(\tau)d\nu(t,x)\right)$, and $[u]$ is the barycentre of $\nu$ in the sense that
	$$
	[u](t,x) = \int_\R \tau d\nu(t,x).
	$$
	From these properties, we recover conditions (1)--(3) of the definition of entropy Young measure solution.
	
	By taking into account that $[\ns{\phi(u)}] \in H^1(\Omega)$, from Proposition \ref{coerenza formulazione discreta} and from equation \ref{less regular} we deduce that $[u]$ and $[\ns{\phi(u)}]$ satisfy
	$$
	\int_0^T \int_{\Omega} [u] \varphi_t - \nabla [\ns{\phi(u)}] \cdot \nabla \varphi dx dt + \int_{\Omega} u_0(x) \varphi(0,x) dx = 0
	$$
	for all $\varphi \in C^1([0,T] \times \overline{\Omega})$ with $\varphi(T,x) = 0$ for all $x \in \Omega$.
	
	We will now derive the entropy condition \ref{entropy inequality} for $[u]$ and $[\ns{\phi}(u)]$.
	Let $g \in C^1(\R)$, $g' \geq 0$, $G(x) = \int_0^{x}g(\phi(\tau))d\tau$, and let $\varphi \in \tests([0,T]\times\Omega)$ with $\varphi \geq 0$.
	Define also
	$$
	G^\star([u]) = \sum_{i=1}^3 \int_0^{S_i\left([\ns{\phi(u)}]\right)} g(\tau)d\tau.
	$$
	By Theorem \ref{parametrized measures}, we have the following equalities
	\begin{eqnarray*}
		- \int_0^T \int_\Omega G^\star([u]) \varphi_t dx dt
		& = & -\sh{\int_0^T \langle \ns{G}(u), \ns{\varphi}_t \rangle dt} \\
		& = & \sh{\int_0^T \langle \ns{G}(u)_t, \ns{\varphi}\rangle dt} \\
		& = & \sh{\int_0^T  \langle\ns{g}(\ns{\phi}(u))u_t, \ns{\varphi}\rangle  dt} \\
		& = & \sh{\int_0^T \langle\ns{g}(\ns{\phi}(u))\Delta_\Lambda \varphi(u), \ns{\varphi}\rangle  dt}.
	\end{eqnarray*}
	By Lemma \ref{lemma entropy condition} and by $[\ns{\phi}(u)] \in H^1(\Omega)$, we deduce
	$$
	\int_0^T \langle\ns{g}(\ns{\phi}(u))\Delta_\Lambda v, \ns{\varphi}\rangle dt
	\sim
	\int_0^T \langle\div^-(\ns{g}(\ns{\phi}(u) \grad^+\phi(u)) - \ns{g}'(\ns{\phi}(u))  |\grad^- \phi(u)|^2, \ns{\varphi} \rangle dt
	$$
	%and, by the discrete summation by parts formula,
	%$$
	%	\int_0^T \langle\div^-(\ns{g}(\phi(u) \grad^+\phi(u)) - \ns{g}'(\phi(u))  |\grad^- \phi(u)|^2, \ns{\varphi} \rangle dt
	%	\sim
	%	\int_0^T \langle-\ns{g}(\phi(u)) \grad^+(\phi(u)),\grad^- \ns{\varphi}\rangle - \langle\ns{g}'(\phi(u))  |\grad^- \phi(u)|^2, \ns{\varphi}\rangle dt
	%$$
	By the discrete summation by parts formula and by Theorem \ref{mainthm},
	\begin{eqnarray*}
		-{\int_0^T\langle\div^-(\ns{g}(\ns{\phi}(u) \grad^+\phi(u)),\ns{\varphi}\rangle}dt
		&\sim&
		{\int_0^T\langle\ns{g}(\ns{\phi}(u)) \grad^+\ns{\phi}(u),\grad^- \ns{\varphi}\rangle}dt\\
		&\sim&
		\int_0^T\int_\Omega g([\ns{\phi}(u)]) \nabla[\ns{\phi}(u)] \cdot \nabla\varphi dx dt
	\end{eqnarray*}
	and, by Proposition 4.3 of \cite{ema2},
	$$
	\sh{\int_0^T \langle\ns{g}'(\ns{\phi}(u))  |\grad^- \ns{\phi}(u)|^2, \ns{\varphi}\rangle dt}
	\geq
	\int_0^T \int_\Omega g'([\ns{\phi}(u)])  |\nabla [\ns{\phi}(u)]|^2 \varphi dx dt.
	$$
	Putting together the above inequalities, we deduce
	$$
	\int_0^T \int_\Omega G^\star([u]) \varphi_t - g([\ns{\phi}(u)]) \nabla([\ns{\phi}(u)]) \cdot \nabla\varphi - g'([\ns{\phi}(u)])  |\nabla [\ns{\phi}(u)]|^2 \varphi dx dt \geq 0,
	$$
	so that $[u]$ and $[\ns{\phi}(u)]$ satisfy the entropy condition \ref{entropy inequality}.
	
	(\ref{vweak}).
	Since $\nu(t,x)$ is Dirac, by Theorem 3.13 of \cite{ema2}, it coincides with $[u]$ and, as a consequence, we also have $[\ns{\phi}(u(t))] = \phi([u])$.
	By substituting $[u]$ and $\phi([u])$ in equation \ref{less regular}, we obtain
	$$
	\int_{0}^T \int_{\Omega} [u](t,x) \varphi_t + \phi([u])(t,x) \Delta \varphi d(t,x) + \int_{\Omega} u_0(x) \varphi(0,x) dx = 0,
	$$
	that is, $([u], [\ns{\phi}(u(t))])$ is a very weak solution of Problem \ref{mark}.
	
	(\ref{weak}).
	In addition to the conclusions of point (\ref{vweak}), we also have $\phi([u]) \in L^\infty((0,T),H^1(\Omega))$, so that by Proposition \ref{coerenza formulazione discreta} the following equality
	$$
	\langle \lap \phi(u(t))(x), \ns{\varphi}(t) \rangle = - \int_{\Omega} \nabla \phi([u])(t) \cdot \nabla \varphi(t) dx
	$$
	holds for a.e.\ $t \geq 0$.
	Hence, by substituting $[u]$ and $\phi([u])$ in equation \ref{less regular}, we deduce
	%$$
	%\int_{0}^T \int_{\Omega} [u](t,x) \varphi_t - \nabla [\phi([u])](t,x) \nabla \varphi d(t,x) + \int_{\Omega} u_0(x) \varphi(0,x) dx = 0,
	%$$
	that %is,
	$[u]$ is a weak solution of Problem \ref{mark}.
	
	We will now prove (\ref{classic}).
	By Theorem %2.14 and by Corollary 2.16 of \cite{ema2}
	\ref{mainthm}, $[u]=\sh{u}$ and $\sh{u} \in C^1(\R_{\geq0}, C^{2}(\Omega))$.
	Moreover, $\sh{\ns{\phi}(u)} = \phi(\sh{u})\in C^1([0,+\infty), C^2(\Omega))$.
	By Theorem 2.15 of \cite{ema2}, $[\Delta_\Lambda \ns{\phi}(u(t))] = \Delta \phi(\sh{u(t)})$ for all $t \geq 0$.
	%Moreover, by Lemma \ref{equiv derivata}, $\D_t u \sim U_t$.
	
	Let $x \in \partial \Omega$ and let $\normal_x$ be the unit exterior normal vector at $x$, and define $A(x)=\{ i : \text{ there exists } y\in\partial_{\Lambda} \dom, y \sim x \text{ such that } i \in I_y^+\cup I_y^-  \}$. We claim that $\normal_x = \sum_{i \in A(x)} a_i e_i$ for some $a_i \in \R$. In fact, if $i\not\in A(x)$, then for all $y\in\partial_\Lambda \dom$ with $y \sim x$, $y\pm\varepsilon e_i \in\dom$. By overspill, there is $\eta \in \R$ such that $x+r e_i \in \Omega$ for all $0\leq|r|<\eta$. As a consequence, $\normal_x \cdot e_i =0$, as desired.
	
	Since $u(t) \in S^2(\Omega)$ for all $t\geq0$, if $y, z \in \dom$ satisfy $y \sim z$, then $\grad u(t,y) \sim \grad u(t,z)$.
	In particular, the boundary conditions \ref{neumann} ensure that for all $y\in\partial_\Lambda \dom$ with $y \sim x$, if $i \in A(x)$, then $\D_i^\pm u(t,y) \sim 0$.
	Hence, for any of such $y$ we have the equalities
		$$
		\frac{\partial \sh{u}(t,x)}{\normal}
		=
		\sh{\left(\grad u(t,y)\right)}\cdot \normal
		=
		\sh{\left(\sum_{i \in A(x)} a_i \D_i^\pm u(t,y)\right)}
		=
		0.
		$$
	Since the above argument holds for any $x \in \partial \Omega$, we deduce
		$$\frac{\partial \phi(\sh{u})}{\partial \normal} = 0 \text{ in } [0,+\infty) \times \partial\Omega.$$
	As a consequence, $\sh{u}$ is a classic global solution of Problem \ref{mark}.
\end{proof}

%In section \ref{sezione riemann}, we will prove that if $u_0(x) \not \in (a,b)$ for all $x \in \Omega$, then $u$ satisfies the hypotheses (\ref{ym}) and (\ref{vweak}) of the previous theorem.
%In this case, $[u]$ is a two-phase Young measure solution of problem \ref{mark} in the sense of \cite{matete}.
%If $u_0 \in C^0((-\infty,a]^k)$ or $u_0 \in C^0([b,+\infty))$, then \color{red} it is easy to prove that $u$ is regular enough to satisfy the hypotheses (\ref{classic}), so that $[u]$ is a classical solution of problem \ref{mark}.

\subsection{The case $u^+ = +\infty$}

We will now discuss coherence of the grid solution to problem \ref{pb discreto} with the measure-valued solution to equation \ref{mark} under the hypothesis that $u^+ = +\infty$.

\begin{theorem}\label{equivalenza smarazzo}
	Let $[u], [\ns{\phi(u)}]$ be the
	grid solution of Problem \ref{pb discreto}, and let $\nu(t,x)$ the Young measure associated to $u$.
	%Define also $\overline{u}(t,x) = \int_{\R} \tau d\nu(t,x) \in L^\infty([0,T],L^1(\Omega))$ and let $\mu = [u]-\overline{u}$.
	If $[\ns{\phi(u)}] \in L^2([0,+\infty),H^1(\Omega))$,
	then $[u], [\ns{\phi(u)}]$ is an entropy Radon measure solution of problem \ref{mark} in the sense of equation \ref{soluzione smarrazzo}.
\end{theorem}
\begin{proof}
	%Let $\nu(t,x)$ the Young measure associated to $u$,
	Let $$u_r(t,x) = \int_{\R} \tau d\nu(t,x)$$ be the barycentre of $\nu$, and let $\mu(t) = [u](t)-u_r(t)$.
	The Young measure $\nu(t,x)$ corresponds to the regular term of the solution to problem \ref{mark}, and the Radon measure $\mu$ corresponds to the singular term.
	
	The hypothesis $[\ns{\phi(u)}](t) \in H^1(\Omega)$ ensures that $[\ns{\phi(u)}](t,x)$ is single-valued for a.e.\ $t \geq 0$ and $x \in \Omega$.
	If $[\ns{\phi(u)}](t,x) = c \not = 0$, this implies that $\nu(t,x)$ is a superposition of at most two Dirac measures centred at $S_1(c)$ and at $S_2(c)$.
	If $[\ns{\phi(u)}](t,x) = 0$, then $\nu(t,x)$ is a Dirac Young measure centred at $0$.
	
	Notice that, for any $\varphi \in C^1([0,T] \times \overline{\Omega})$ with $\varphi(T,x) = 0$ for all $x \in \Omega$, we have the equality
	$$
	\int_0^T \langle u, \ns{\varphi} \rangle dt = \int_0^T \int_{\Omega} [u] \varphi dxdt =  \int_0^T \int_{\Omega} u_r \varphi dx dt + \int_{0}^{T} \ldual \mu,\varphi\rdual dt,
	$$
	for any arbitrary $T > 0$.
	By Proposition \ref{coerenza formulazione discreta}, by equality \ref{less regular} and by the hypothesis that $[\ns{\phi(u)}](t) \in H^1(\Omega)$, we deduce that $u_r$, $\mu$ and $[\ns{\phi(u)}]$ satisfy the equality
	$$
	\int_{0}^T \ldual \mu, \varphi_t\rdual dt + \int_0^T \int_{\Omega} u_r \varphi_t - \nabla [\ns{\phi(u)}] \cdot \nabla \varphi dx dt + \int_{\Omega} u_0(x) \varphi(0,x) dx = 0,
	$$
	so that $[u]$ and $[\ns{\phi}(u)]$ induce a Radon entropy solution to problem \ref{mark} in the sense of equation \ref{soluzione smarrazzo}.
	
	The entropy condition under the hypothesis that $G$ is equi-integrable can be deduced from Lemma \ref{lemma entropy condition} and from an argument analogous to the one in the proof of point (\ref{ym}) of Theorem \ref{mina}.
\end{proof}

We can also prove that the singular part of the Radon measure solution can be disintegrated as in equation \ref{singular disintegration}. 

\begin{proposition}\label{equivalenza smarazzo 2}
	Let $\mu$ be defined as in the proof of Theorem \ref{equivalenza smarazzo}.
	There exists a function $\tilde{\mu} : L^\infty([0,+\infty), \prob{(\Omega)})$ such that
	$$
	\mu(t) = \left(\int_{\Omega} u_0(x) dx - \int_{\Omega} u_r(t,x) dx \right) \tilde{\mu}(t).
	$$
	Moreover, the support of $\tilde{\mu}(t)$ is a null-set with respect to the $k$-th dimensional Lebesgue measure for all $t\geq0$. 
\end{proposition}
\begin{proof}
	%We will now prove that $\mu$ satisfies equation \ref{singular disintegration}.
	By definition of $\mu$, $\mu$ can be interpreted as a function $\mu: L^\infty([0,+\infty),\rad{(\Omega)})$ that satisfies
	$$
	\int_{\Omega} u_r(t,x) dx+\int_{\Omega} d\mu(t) = \sh{\norm{u(t,x)}_1}.
	$$
	By Theorem \ref{esistenza}, $\sh{\norm{u(t,x)}_1} = \int_{\Omega} u_0(x) dx$, so the first part of the assertion is proved.
	The second part of the assertion is a consequence of $\norm{u(t)}_1 \in \fin$.
\end{proof}

\section{Asymptotic behaviour of the grid solutions to the ill-posed PDE} \label{properties}

In this section, we will draw conclusions about the asymptotic behaviour of the grid solutions to problem \ref{mark} by studying the asymptotic behaviour of the solutions to the grid function formulation \ref{pb discreto}.
In particular, we will carry out this study by determining the stability of the steady states of problem \ref{pb discreto}.

A steady state of problem \ref{pb discreto} is a grid function $\steady \in \grid{\Omega}$ that satisfies $\lap\ns{\phi}(\steady) = 0$.
It can be proved that, by definition of $\lap\ns{\phi}$ and thanks to the boundary conditions \ref{neumann}, $\steady$ is a steady state if and only if $\ns{\phi}(\steady(x)) = c$ for all $x \in \dom$.

\begin{lemma}\label{lemma grafi}
$\lap \ns{\phi(\steady)}=0$ if and only if there exists $c \in \hR$ such that $\ns{\phi}(\steady(x))=c$ for all $x \in \dom$.
\end{lemma}
\begin{proof}
For a matter of commodity, let $|\dom|=d$ and $\dom = \{x_1, \ldots, x_d\}$.

Let $A=(a_{i,j})$ be the matrix associated to the linear operator $-\lap : \grid{\Omega}\rightarrow \grid{\Omega}$ and denote by $r$ the rank of $A$.
Since $\lap(\steady)=0$ whenever there exists $c \in \hR$ such that $\ns{\phi}(\steady(x))=c$ for all $x \in \dom$, in order to prove the desired equivalence it is sufficient to prove that $r=d-1$.

Let $G=(V,E)$ be a graph with vertex set $V=\{1, \ldots, v\}\subseteq \ns{\N}$.
Recall that the Laplacian matrix of a graph $G=(V,E)$ is defined as the matrix $L=(l_{i,j})$ such that $l_{i,i}$ is the degree of $i$ and $l_{i,j}=-1$ if and only if $\{i,j\}\in E$.
Notice that, by definition, the Laplacian matrix of a graph is a square symmetric matrix of order $v$.

Consider the graph $G=(\dom,E)$, where $E=\{\{x_i,x_j\}: \norm{x_i-x_j}_1=\varepsilon\}$.
By definition of $\lap$ and thanks to the boundary conditions \ref{neumann}, $\varepsilon^2A$ is the Laplacian matrix of $G$.
By Lemma 3 of \cite{graph}, the multiplicity of $0$ as eigenvalue of $\varepsilon^2 A$ is equal to $d-g$, where $g$ denotes the number of connected components of $G$.
Since $\Omega$ is connected, $g=1$, so that $r=d-1$, as desired.
\end{proof}

As a consequence, a steady state $\steady$ can assume up to three values $\omega_1 \in (0, u^-]$, $\omega_2 \in (u^-, u^+)$ and, when $u^+ < +\infty$, $\omega_3 \in [u^+, +\infty)$ satisfying $\phi(\omega_1) = \phi(\omega_2) = \phi(\omega_3)$.
By Proposition \ref{coerenza formulazione discreta}, the steady states of the grid function formulation \ref{pb discreto} induce a steady state for problem \ref{mark}.

%, and if $u_0 \in L^\infty(\Omega)$ is a steady state for problem \ref{mark}, then $\ns{u_0}$ is a steady state of the grid function formulation \ref{pb discreto}.

Notice however that a steady state of problem \ref{mark} corresponds to a grid function $\widetilde{v}$ that satisfies only the weaker condition $\lap\ns{\phi}(\widetilde{v}) \sim 0$.
If $\norm{\widetilde{v}}_\infty \in \fin$, then there exists a steady state $\steady$ of problem \ref{pb discreto} with $\norm{\steady-\widetilde{v}}_\infty\sim0$.
In this case, the stability of $\widetilde{v}$ can be determined by studying the stability of $\steady$: if $\tilde{u}$ is the only asymptotically stable equilibrium infinitesimally close to $\tilde{v}$, then $\tilde{v}$ is in the basin of attraction of $\tilde{u}$ and, as a consequence, $v$ is an asymptotically stable steady state of problem \ref{pb discreto}.
If, on the other hand, $\widetilde{v}$ is infinitesimally close to some steady states that are not asymptotically stable, then problem \ref{pb discreto} with initial condition $\tilde{v}$ might evolve towards an asymptotically stable steady state $\tilde{w} \not \equiv \tilde{v}$.
In this case, $\tilde{v}$ would not be a stable steady state of problem \ref{mark}.

Under the hypothesis $u^+ = +\infty$ and $\norm{\widetilde{v}}_\infty \not \in \fin$, then $\widetilde{v}$ induces a measure-valued steady state of problem \ref{mark}, but there might not exist a steady state $\steady$ of the grid function formulation \ref{pb discreto} which satisfies $\norm{\steady-\widetilde{v}}_\infty\sim0$.
Nevertheless, in section \ref{asymptotic +infty}, we will show that the asymptotic behaviour of the grid solutions to problem \ref{mark} can be characterized a posteriori from the asymptotic behaviour of the solutions to problem \ref{pb discreto}.

\subsection{Asymptotic behaviour of the solutions to problem \ref{pb discreto}}

%In order to study the asymptotic behaviour of the solutions to problem \ref{pb discreto}, we need to study the stability of its steady states.
Since problem \ref{pb discreto} corresponds to a hyperfinite dynamical system, we need to introduce an appropriate notion of stability for its steady states.
Our choice is to use the nonstandard counterpart of the classical notion of stability in the $L^\infty$ norm for discrete dynamical systems.
In the following definition, it is useful to keep in mind that $u \in \grid{\Omega}$ can be identified with a vector in the euclidean space $\hR^{|\dom|}$.

\begin{definition}
	Let $f: \grid{\Omega} \rightarrow \ns{\R}$ and let $v(t) : \ns{\R} \rightarrow \grid{\Omega}$ be the solution of the nonstandard differential equation $u' = f(u)$ with initial data $v(0)$.
	We will say that $u \in \grid{\Omega}$ is
	\begin{itemize}
		\item \stable\ iff for all $\eta \in \ns{\R}, \ \eta > 0$ there exists $\delta \in \ns{\R}, \ \delta > 0$ such that % for all $w \in \ns{\R^k}$
		$\norm{u-v(0)}_\infty < \delta$ implies $\norm{u-v(t)}_\infty<\eta$ for all $t \in \ns{\R_+}$;
		\item $\ns{}$attractive iff there exists $\rho \in \ns{\R}, \ \rho > 0$ such that %for all $w \in \ns{\R^k}$
		$\norm{u-v(0)}_\infty < \rho$ implies $\ns{\lim_{t \rightarrow +\infty}}\norm{u-v(t)}_\infty = 0$;
		\item asymptotically \stable \ iff it is \stable\ and $\ns{}$attractive;
		\item globally asymptotically \stable \ iff it is \stable\ and for all $v(0) \in \dom$ $\ns{\lim_{t \rightarrow +\infty}}\norm{u-v(t)}_\infty = 0$;
		\item \unstable \ iff it is not \stable.
	\end{itemize}
	Notice that a necessary condition for $u$ to be \stable \ or $\ns{}$attractive is that $f(u) = 0$, i.e.\ $u$ must be an equilibrium point of the differential equation.
\end{definition}

%If $|\dom|$ is finite, then the above notions coincide with the nonstandard extensions of the classical notions of stability for a dynamical system.
Since the $\ns{L^\infty}$ norm over $\dom$ is equivalent to the euclidean norm in $\hR^{|\dom|}$, the stability in the $\ns{L^\infty}$ norm for the grid function formulation \ref{pb discreto} can be studied by exploiting the theory of finite dynamical systems.

For the following analysis of the asymptotic behaviour of solutions of system \ref{pb discreto}, we assume that the steady states are isolated in $\simplex^+(\ns{u_0})$, i.e.\ that there is only a hyperfinite number of steady states in $\simplex^+(\ns{u_0})$.
For a discussion of this hypothesis and for sufficient conditions that ensure the existence of a hyperfinite number of steady states in $\simplex^+(\ns{u_0})$, we refer to Lizana and Padron \cite{sd}.
Their hypothesis is a sharpening of the condition that $S_1', S_2'$ and $S_3'$ must be linearly independent on the spinoidal interval $(u^-,u^+)$, already discussed in \cite{stable patterns}.

\begin{proposition}\label{limiti}
	If the steady states of \ref{pb discreto} are isolated in $\simplex^+(\ns{u_0})$ and if $M$ is the largest positively invariant set contained in
	$$\simplex^+(\ns{u_0})\cap\left\{ f \in \grid{\Omega} : \phi(f(x)) \text{ is constant} \right\},$$
	then %for all $\varepsilon \in \ns{\R}$ there exists $t_\varepsilon \in \ns{\R}$ such that for all $t > t_\varepsilon$ $$\inf_{p \in M}|u(t) - p|<\varepsilon.$$
	$\ns{\lim_{t \rightarrow +\infty}} u(t) \in M$.
	In particular, system \ref{pb discreto} has at least an asymptotically \stable \ steady state.
\end{proposition}
\begin{proof}
	This is a consequence of Proposition 2 of \cite{sd}.
\end{proof}

We observe that, under the hypothesis that the steady states of system \ref{pb discreto} are isolated in $\simplex^+(\ns{u_0})$, then $\ns{}$stability is equivalent to asymptotic $\ns{}$stability.

\begin{lemma}\label{lemma stable implies as}
	If $\steady$ is a \stable \ steady state of system \ref{pb discreto} and if the steady states are isolated in $\simplex^+(\ns{u_0})$, then $\steady$ is asymptotically \stable.
\end{lemma}
\begin{proof}
	Suppose that $\steady$ is \stable: since the steady states of system \ref{pb discreto} are isolated, we can find $\rho > 0$ such that if $0 < \norm{\steady-v}_\infty < \rho$ then $v$ is not a steady state of system \ref{pb discreto}.
	By the $\ns{}$stability of $\steady$, we can find $\delta > 0$ such that if $\norm{\steady-v}_\infty < \delta$ then, denoting by $v(t)$ the solution of system \ref{pb discreto} with initial data $v$, $\norm{\steady-v(t)}_\infty < \rho$ for all $t \in \ns{\R_+}$.
	Moreover, by Proposition \ref{limiti} $v(t)$ converges to a steady state of system \ref{pb discreto}.
	By our choice of $\rho$, this steady state must be $\steady$, hence $\steady$ is \attractive.
\end{proof}

\subsection{Steady states of problem \ref{pb discreto}}

For a matter of commodity, we will carry out the study of the steady states of problem \ref{pb discreto} in the case where $k = 1$, and where the spatial domain is $[0,1]_\Lambda = [0,\varepsilon, \ldots, N\varepsilon=1]$, but the analysis can be carried out in higher dimension and with other domains, as argued in \cite{sd}.
Moreover, we identify a grid function $u \in \grid{[0,1]_\Lambda}$ with a vector $u\in\hR^{N+1}$, with the convention that $u_i$, the $i$-th component of $u$, satisfies $u_i = u(i\varepsilon)$.
If $u : \hR \rightarrow \grid{[0,1]_\Lambda}$, we will identify it with a vector-valued function $u : \hR \rightarrow \hR^{N+1}$, with the convention that $u_i(t) = u(t,i\varepsilon)$.

We begin the study of the $\ns{}$stability of the steady states of system \ref{pb discreto} by discussing its homogeneous steady state $\eq = \left(\vsum{\ns{u_0}}, \ldots, \vsum{\ns{u_0}}\right)$.

\begin{proposition}\label{poutpurri}
	The homogeneous steady state $\eq$ of system \ref{pb discreto} has the following properties:
	%If $u$ is a solution of \ref{pb discreto}, then the following are satisfied:
	\begin{itemize}
		\item if $\vsum{\ns{u_0}} < u^-$ or $\vsum{\ns{u_0}} > u^+$, then $\eq$ is \stable;
		\item %if $\max_{0 \leq i \leq N} v_i < u^-$ or $\min_{0 \leq i \leq N} v_i > u^+$, i.e.\
		if $\eq$ is the only steady state of \ref{pb discreto}, then $\eq$ is globally asymptotically \stable;
		\item
		if $u^- < \vsum{\ns{u_0}} < u^+$, then $\eq$ is \unstable.
		Moreover, if $\ns{u_0} \not \sim \eq$ and if the steady states are isolated in $\simplex(\ns{u_0})$, then $u$ converges to a non-homogeneous steady state.
	\end{itemize}
\end{proposition}
\begin{proof}
	It is a consequence of %the Transfer Principle, of
	Proposition 3 and of Corollary 4 of \cite{sd}.
\end{proof}

%\begin{corollary}
%When the initial data $v$ satisfies
%$\max_{0 \leq j \leq N} v_j < u^-$ or $\min_{0 \leq j \leq N} v_j > u^+$,
%the dynamics of system \ref{pb discreto} are the usual dynamics of parabolic smoothing.
%\end{corollary}

In addition to the homogeneous steady state $\eq$, system \ref{pb discreto} may have many non-homogeneous steady states.
If we denote by $n_i$ the number of components of $\steady$ that assume the value $\omega_i$, by Proposition \ref{esistenza} we obtain the relations
$$
n_3 = N+1-(n_1+n_2),\ n_1 \omega_1 + n_2 \omega_2 + (N+1-(n_1+n_2))\omega_3 = (N+1)\vsum{\ns{u_0}}
$$
that in the case where $u^+ = +\infty$ become
\begin{eqnarray} \label{numero soluzioni}
n_2 = N+1-n_1, & & n_1 \omega_1 + (N+1-n_1) \omega_2 = (N+1)\vsum{\ns{u_0}}.
\end{eqnarray}

In the first step of the study of the $\ns{}$stability of the non-homogeneous steady states of system \ref{pb discreto}, we will prove that all the steady states with $n_2 >1$ are \unstable.

\begin{proposition}\label{stabilita 1}
	If $\steady \in \ns{\R^{N+1}}$ is a steady state of \ref{pb discreto} with $n_2 >1$, then it is \unstable.
\end{proposition}
\begin{proof}
	As in the proof of Proposition 4 of Witelski, Schaeffer and Shearer \cite{discrete 2}, we will show that $\steady$ is not a stable steady state of \ref{pb discreto} by showing that it is not a local minimum of a suitable Lyapunov function: the thesis follows from this result.
	In order to simplify the notation, suppose that $n_3 = 0$, as the proof for the general case can be deduced by the argument below.
	
	Consider the perturbed steady state given by
	\begin{equation*}
	\left\{
	\begin{array}{rcll}
	u_{i_1}(t) & = & \omega_2 +q&\\
	u_{i_2}(t) & = & \omega_2 -q&\\
	u_{i_k}(t) & = & \omega_2 & \mathrm{for \ } k = 3, 4, \ldots, n_2\\
	u_i(t) & = & \omega_1 & \mathrm{otherwise}
	\end{array}
	\right.
	\end{equation*}
	Let now $V(u_i) = \int_0^{u_i} \phi(s) ds$, and $L(u) = \sum_{i = 0}^{N+1} V(u_i)$.
	From Proposition 4 of \cite{discrete 2}, it can be deduced that $L$ is a Lyapunov function for system \ref{pb discreto}.
	By evaluating $L$ as a function of $q$, we get
	$$
	L(q) = \frac{V(\omega_2 +q) + V(\omega_2 -q) + (n_2-2) V(\omega_2) + (N+1-n_2)V(\omega_1)}{N}
	$$
	so we deduce
	\begin{eqnarray*}
		\left.\frac{dL}{dq} \right|_{0} = 0 & \mathrm{and} & \left.\frac{d^2L}{dq^2}\right|_{0} = \frac{2}{N}\phi'(\omega_2) < 0,
	\end{eqnarray*}
	where the last inequality follows from the hypothesis that $\omega_2 \in (u^-,u^+)$.
	We conclude that $\steady$ is not a local minimum of $L$ and, as a consequence, that $\steady$ is \unstable.
\end{proof}

The characterization of the asymptotically \stable\ non-homogeneous steady states of system \ref{pb discreto} is based on the following bound on $\phi'(\omega_2)$.

\begin{lemma}\label{lemma bound}
	If $\steady$ is an asymptotically \stable\ non-homogeneous steady state of \ref{pb discreto} with $n_2=1$, then it holds the inequality
	\begin{equation}\label{inequality}
	|\phi'(\omega_2)|  < \frac{\max\{\phi'(\omega_1), \phi'(\omega_3)\}^2}{N\min\{\phi'(\omega_1), \phi'(\omega_3)\}}.
	\end{equation}
\end{lemma}
\begin{proof}
	For a matter of commodity, suppose that
	\begin{equation*}
	\left\{
	\begin{array}{rcll}
	\steady_0 & = & \omega_2&\\
	\steady_i & = & \omega_1& \mathrm{for \ } i = 1, 2, \ldots, n_1\\
	\steady_i & = & \omega_3& \mathrm{otherwise.}\\
	\end{array}
	\right.
	\end{equation*}
	Let $X_1(\steady) = -(\phi'(\steady_0)+\phi'(\steady_1))$ and define by recursion
	$$
	X_{i+1}(\steady) = -\phi'(\steady_{i+1})X_i(\steady) + (-1)^{i+1}\prod_{j = 0}^i \phi'(\steady_j)
	$$
	It is a consequence of Proposition 8 of \cite{sd} that asymptotic $\ns{}$stability of $\steady$ is equivalent to $(-1)^i X_i(\steady) > 0$
	for $i = 1, \ldots, N$.
	Notice that, as long as $i \leq n_1$, $$X_i(\steady) = (-1)^i \phi'(\omega_1)^{i-1}(\phi'(\omega_1)+i\phi'(\omega_2)),$$
	so that $(-1)^i X_i(\steady) > 0$ is equivalent to
	$$
	|\phi'(\omega_2)| < \frac{\phi'(\omega_1)}{i}.
	$$
	Since the above inequality is true for all $i\leq n_1$, we obtain
	$|\phi'(\omega_2)| \leq \frac{\phi'(\omega_1)}{n_1}$.
	For $i = n_1 + 1, \ldots, N$, a similar computation shows that $(-1)^i X_i(\steady) > 0$ implies
	$$
		|\phi'(\omega_2)| < \frac{\phi'(\omega_1)\phi'(\omega_3)}{n_1\phi'(\omega_3)+(i-n_1)\phi'(\omega_1)};
	$$
	since this inequality is true for all $n_1 < i \leq N$, we deduce
	$$
	|\phi'(\omega_2)|
	\leq 
	\frac{\phi'(\omega_1)\phi'(\omega_3)}{n_1 \phi'(\omega_3)+(N-n_1)\phi'(\omega_1)}.
	$$
	From the inequality
	$$
	\frac{\phi'(\omega_1)\phi'(\omega_3)}{n_1 \phi'(\omega_3)+(N-n_1)\phi'(\omega_1)}
	\leq  \frac{\max\{\phi'(\omega_1), \phi'(\omega_3)\}^2}{N\min\{\phi'(\omega_1), \phi'(\omega_3)\}}
	$$
	we deduce that the desired result holds.
\end{proof}

\subsection{Asymptotic behaviour of the grid solutions under the hypothesis $u^+ < +\infty$}

We will now discuss the asymptotic behaviour of the grid solutions to problem \ref{mark} under the hypothesis that $u^+ < +\infty$.
Under this hypothesis, the steady states of the grid function formulation with $n_2 = 0$ are all asymptotically \stable.

\begin{proposition}\label{n_2 = 0}
	Let $\steady$ be a steady state of system \ref{pb discreto} with $n_2 = 0$.
	Then $\steady$ is asymptotically \stable.
\end{proposition}
\begin{proof}
	It is a consequence of Proposition 8 of \cite{sd}.
\end{proof}

It turns out that, thanks to the hypotheses over $\phi$, all \stable\ non-homogeneous steady states of system \ref{pb discreto} for which $\omega_1 \not \sim u^-$ and $\omega_3 \not \sim u^+$ must have $n_2 = 0$, giving a partial converse to Proposition \ref{n_2 = 0}.

\begin{proposition}\label{last stability}
	If $\steady$ is an asymptotically \stable\ non-homogeneous steady state of \ref{pb discreto} with
	$\omega_1 \not \sim u^-$ and $\omega_3 \not \sim u^+$,
	%$$\min\{\phi'(q_1), \phi'(q_3)\} \not \sim 0$$
	then $n_2 = 0$.
\end{proposition}
\begin{proof}
	Suppose towards a contradiction that $n_2 = 1$.
	The hypotheses $\omega_1 \not \sim u^-$ and $\omega_3 \not \sim u^+$ imply $\min\{\ns{\phi}'(\omega_1), \ns{\phi}'(\omega_3)\} \not \sim 0$, otherwise either $\phi'(\sh{\omega_1})=0$ or $\phi'(\sh{\omega_3})=0$, against the hypotheses \ref{ipotesi di partenza}.
	As a consequence, $N\min\{\ns{\phi}'(\omega_1), \ns{\phi}'(\omega_3)\}$ is infinite.
	Thanks to inequality \ref{inequality}, we deduce that $|\ns{\phi}'(\omega_2)| \sim 0$.
	By the hypotheses over $\phi$, there exists $\omega_2 \in \ns{(u^-,u^+)}$ with $|\ns{\phi}'(\omega_2)| \sim 0$ if and only if $\omega_2 \sim u^-$ or $\omega_2 \sim u^+$.
	However, $\omega_2 \sim u^-$ implies $\omega_1 \sim u^-$ and $\omega_2 \sim u^+$ implies $\omega_3 \sim u^+$, in contradiction with the hypotheses $\omega_1 \not \sim u^-$ and $\omega_3 \not \sim u^+$.
\end{proof}

Putting together the results of this section, we can characterize the asymptotic behaviour of a grid solution of problem \ref{mark}.
In particular, for almost every initial data, the grid solution converges to a steady state that is a superposition of at most two Dirac measures centred at the stable branches of $\phi$.

\begin{proposition}
	Let $[u], [\ns{\phi(u)}]$ be the grid solution of problem \ref{mark} with initial data $\ns{u_0}$.
	For almost every $u_0 \in L^\infty(\Omega)$, $[u]$ converges to a steady state $\nu$ satisfying:
	\begin{enumerate}
		\item there exists $c \in \R$ such that $\int_{\R} \phi(\tau) d\nu(x) = c$ for all $x \in \Omega$;
		\item there exist $r_1 \in [0,u^-]$, $r_3 \in [u^+,+\infty)$, and $\lambda_1, \lambda_3 : \Omega \rightarrow [0,1]$, such that
		\begin{enumerate}
			\item $ \nu(x) = \lambda_1(x) \delta_{r_1}+\lambda_3(x) \delta_{r_3}$ for a.e.\ $x \in \Omega$;
			\item $ \phi(r_1) = \phi(r_3) = c$;
			\item $\lambda_1(x)+\lambda_3(x) = 1$ for a.e.\ $x \in \Omega$.
		\end{enumerate}
	\end{enumerate}
\end{proposition}
\begin{proof}
	Notice that, since $\phi(u_0(x)) \in \R$ for all $x \in \Omega$,
	$\lap\ns{\phi}(u_0(x)) \sim 0$ for all $x \in \dom$ if and only if %$\phi(u_0(x))$ is constant, i.e.\ if and only if there are $r_1 \in [0,u^-]$, $r_2 \in (u^-,u^+)$ and $r_3 \in [u^+,+\infty)$ satisfying  $\phi(r_1)=\phi(r_2)=\phi(r_3)$ such that for all $x \in \Omega$ there exists $i \leq 3$ for which $u_0(x)=r_i$.
	%In this case it also holds
	$\lap\ns{\phi}(u_0(x)) = 0$ for all $x \in \dom$.
	This condition is verified if and only if there are $r_1 \in [0,u^-]$, $r_2 \in (u^-,u^+)$ and $r_3 \in [u^+,+\infty)$ satisfying  $\phi(r_1)=\phi(r_2)=\phi(r_3)$ such that for all $x \in \Omega$ there exists $i \leq 3$ for which $u_0(x)=r_i$.
	However, the set of such initial data is a null set in $L^\infty(\Omega)$.
	% or there exists $x \in \dom$ such that $\ns{\phi(u_0(x)) \not \sim 0$. In other words, $\ns{u_0}_{|\Omega_{\Lambda}}$ is either an asymptotically $\ns $stable state of system \ref{pb discreto} or

	If $\phi(\ns{u_0}(x))$ is not constant, Proposition \ref{limiti} ensures that a solution of system \ref{pb discreto} with initial data $\ns{u_0}$ will converge towards an asymptotically $\ns $stable steady state $\steady$ of \ref{pb discreto}.
	Thanks to Lemma \ref{lemma grafi}, for each of these steady states there exists $\xi \in \hR$ such that $\ns{\phi}(\steady(x)) = \xi$ for all $x \in \dom$; moreover, $\xi \in \fin$ by Proposition \ref{ex invariant set}.
	As a consequence, equality (1) is readily obtained from Theorem \ref{parametrized measures} by taking $c=\sh{\xi}$.
	
	Thanks to Proposition \ref{stabilita 1}, $\ns $stability of the steady state $\steady$ implies that $0 \leq n_2 \leq 1$, i.e.\ that $[\phi(\steady)]\not \in (u^-,u^+)$.
	Taking into account Theorem \ref{mina}, we obtain (2a) and (2b) with $r_1 = \sh{\omega_1}$ and $r_3=\sh{\omega_3}$.
	Property (2c) is a consequence of conservation of mass settled in Theorem \ref{esistenza}.
\end{proof}

\subsection{Asymptotic behaviour of the grid solutions under the hypothesis $u^+ = +\infty$}\label{asymptotic +infty}

If $u^+ = +\infty$, the bound of Lemma \ref{lemma bound} becomes
\begin{equation}\label{inequality2}
|\phi'(\omega_2)|  < \frac{\phi'(\omega_1)}{N}.
\end{equation}
From this inequality we will deduce that a necessary condition for the asymptotic $\ns{}$stability of a non-homogeneous steady state $p$ is that $\phi'(\omega_2) \sim 0$, and this is possible only when $\omega_2$ is infinite.

\begin{proposition}\label{spike}
	Suppose that $u^+ = +\infty$ and that $\steady$ is an asymptotically \stable\ non-homogeneous steady state of \ref{pb discreto}.
	Then $\omega_2$ is infinite.
\end{proposition}
\begin{proof}
	Since the steady state is non-homogeneous, $\omega_2 > u^-$.
	By inequality \ref{inequality2}, and since $\phi \in C^1(\R)$, it must hold $\ns{\phi}'(\omega_2) \sim 0$.
	Since hypotheses \ref{ipotesi di partenza} entails the inequality $\ns{\phi}'(x)<0$ for all $x>u^-$, the condition $\ns{\phi}'(\omega_2) \sim 0$ is satisfied only if $\omega_2$ is infinite, as desired.
\end{proof}

This result together with Proposition \ref{stabilita 1} implies that any non-homogeneous asymptotically \stable \ steady states of system \ref{pb discreto} in the case where $u^+ = +\infty$ are piecewise constant with a single spike.
Proposition \ref{spike} implies that an infinite amount of the mass is concentrated in the spike.
This result is in accord with both the theoretical results and the numerical experiments of \cite{painter2, reinforcement, sd, sobolev padron, discrete 2}.
However, as we observed previously, we do not expect that $[u]$ converges to a steady state which satisfies Proposition \ref{spike}.
Proposition \ref{spike} should be interpreted as a confirmation of a conjecture by Smarazzo that, for a grid solution $[u]$ of problem \ref{mark}, the regular part of the solution eventually vanishes, and the singular part of the solution prevails.
More precisely, we obtain the following result.

\begin{proposition}\label{proposition conjecture smarrazzo}
	Let $[u], [\ns{{\phi(u)}}]$ be the grid solution of problem \ref{mark} with initial data $\ns{u_0}$.
	For almost every $\ns{u_0} \in L^\infty(\Omega)$, $[u]$ converges to a steady state $\nu+\mu$ satisfying:
	\begin{enumerate}
		\item $\nu$ is a homogeneous Dirac Young measure centred at $0$, i.e.\ $\nu \in L^\infty(\Omega)$ and $\nu(x) = 0$ a.e.;
		\item $\mu= \norm{\ns{u_0}}_1\tilde{\mu}$, and $\tilde{\mu}$ is a probability measure over $\Omega$.
	\end{enumerate}
	In particular, for almost every initial data $\ns{u_0}$, $[u]$ converges to a steady state with null regular part.
\end{proposition}
\begin{proof}
	Notice that, since $\phi(u_0(x)) \in \R$ for all $x \in \Omega$,
	$\lap\ns{\phi}(u_0(x)) \sim 0$ for all $x \in \dom$ if and only if %$\phi(u_0(x))$ is constant, i.e.\ if and only if there are $r_1 \in [0,u^-]$, $r_2 \in (u^-,u^+)$ and $r_3 \in [u^+,+\infty)$ satisfying  $\phi(r_1)=\phi(r_2)=\phi(r_3)$ such that for all $x \in \Omega$ there exists $i \leq 3$ for which $u_0(x)=r_i$.
	%In this case it also holds
	$\lap\ns{\phi}(u_0(x)) = 0$ for all $x \in \dom$.
	This condition is verified if and only if there are $r_1 \in [0,u^-]$ and $r_2 \in (u^-,+\infty)$ satisfying  $\phi(r_1)=\phi(r_2)$ such that for all $x \in \Omega$ there exists $i \leq 2$ for which $u_0(x)=r_i$.
	However, the set of such initial data is a null set in $L^\infty(\Omega)$.
	
	If $\phi(\ns{u_0}(x))$ is not constant, Proposition \ref{limiti} ensures that a solution of system \ref{pb discreto} with initial data $\ns{u_0}$ will converge towards an asymptotically $\ns $stable steady state $\steady$ of \ref{pb discreto}.
	%At the beginning of this section we have argued that there exists $\omega_1 \in (0,u^-)$ and $\omega_2 \in (u^-, +\infty)$ such that $\ns{\phi}(\steady(\omega_1)) = \ns{\phi}(\steady(\omega_2))$ and such that for all $x \in \dom$ there exists $i \leq 2$ such that $\steady(x) = \omega_i$.
	%Thanks to Proposition \ref{spike}, we deduce that $\xi \sim 0$.
	%As a consequence, we deduce  (1) is readily obtained by taking $c=\sh{\xi}$.
	Thanks to Proposition \ref{stabilita 1} and to Proposition \ref{spike}, $\ns $stability of the steady state $\steady$ implies that $n_2 = 1$, that $\omega_2$ is infinite and, as a consequence of the equality $\ns{\phi}(\steady(\omega_1)) = \ns{\phi}(\steady(\omega_2))$ and of the hypotheses over $\phi$, that $\omega_1 \sim 0$.
	Taking into account Theorem \ref{equivalenza smarazzo} and Proposition \ref{equivalenza smarazzo 2}, we obtain the desired assertions.
\end{proof}

\section{The Riemann problem}\label{sezione riemann}

In the study of problems \ref{mark} and \ref{mark II} in the case when $u^+ < +\infty$, the dynamics of solutions with Riemann initial data are of particular interest both in the theoretical and in the numerical setting (see for instance \cite{riemann, lamascia}).
We will discuss the Riemann problem where the initial data $\ns{u_0}$ satisfies
\begin{equation}\label{riemann data}
\ns{u_0}(i\varepsilon) = 	\left\{
\begin{array}{ll}
p_l \in [0, u^-] &\mathrm{for\ } 0 \leq i \leq n\\
p_r \in [u^+, +\infty) &\mathrm{for\ } n+1 \leq i \leq N
\end{array} 
\right.
\end{equation}
for some $n \leq N$.
In order to understand the evolution of system \ref{pb discreto} with initial data \ref{riemann data}, we need to focus on the behaviour of the solution near the discontinuity in the data.
In particular, we will discuss the conditions at which $u_i(t) \in (0, u^-]$, $u_{i+1} \in [u^+, +\infty)$ and either $u_i(t+\tau) \in (u^-,u^+)$ or $u_{i+1}(t+\tau) \in (u^-,u^+)$ for some small $\tau > 0$.
If $u_i(t) \in (0, u^-]$ and $u_i(t+\tau) \in (u^-,u^+)$ we will say that there is an upward phase transition at $u_i(t)$; if $u_{i+1}(t) \in [u^+, +\infty)$ and $u_{i+1}(t+\tau) \in (u^-,u^+)$ we will say that there is a downward phase transition at $u_{i+1}(t)$.

\begin{proposition}\label{riemann transitions}
	Let $u$ be a solution of system \ref{pb discreto} with initial data \ref{riemann data}.
	Then an upward phase transition occurs at $u_i(t)$ for some $t > 0$ and for some $0 \leq i \leq N$ iff $u_i(t) = u^-$,
	\begin{equation}\label{inuguaglianza su}
	\ns{\phi}(u_{i-1}(t)) + \ns{\phi}(u_{i+1}(t)) > 2\ns{\phi}(u^-)
	\end{equation}
	and
	\begin{equation}\label{maximality}
	i = \max_{j \in\ns{\N},\ j \leq N}\{j : u_m(t) \leq u^- \mathrm{\ for\ all\ } m \leq j \}.
	\end{equation}
	A downward phase transition occurs at time $t$ at some $0 \leq i \leq N$ iff $u_i(t) = u^+$,
	\begin{equation}\label{inuguaglianza giu}
	\ns{\phi}(u_{i-1}(t)) + \ns{\phi}(u_{i+1}(t)) < 2\ns{\phi}(u^+)
	\end{equation}
	and
	\begin{equation}\label{minimality}
	i = \max_{j \in\ns{\N},\ j \leq N}\{j : u_m(t) \geq u^+ \mathrm{\ for\ all\ } m \leq j \}.
	\end{equation}
\end{proposition}
\begin{proof}
	Suppose that $u_i(t) = u^-$ for some $t \in \ns{\R}_+$ and for some $i \leq N$.
	Then there is a phase transition iff
	$$
	\varepsilon^2 u'_i(t) 
	= 
	\ns{\phi}(u_{i-1}(t)) - 2\ns{\phi}(u_i(t))+ \ns{\phi}(u_{i+1}(t)) > 0
	$$
	from which \ref{inuguaglianza su} follows.
	Inequality \ref{inuguaglianza giu} can be proved in a similar way.
	Notice that the two inequalities imply that if at time $t$ $u_{i}(t)$, $u_{i+1}(t)$ and $u_{i-1}(t)$ are in the same stable phase, then $u_i$ cannot have a transition at time $t$.
	This is sufficient to entail \ref{maximality} and \ref{minimality} for Riemann initial data.
\end{proof}

\begin{proposition}\label{interfaccia sottile}
	Let $u$ be a solution of system \ref{pb discreto} with initial data \ref{riemann data}.
	For every $t \in \ns{\R}_+$, there exists at most one $i \leq N$ such that $u_i(t) \in (u^-, u^+)$.
\end{proposition}
\begin{proof}
	Conditions \ref{maximality} and \ref{minimality} imply that if $u_i(t)$ and $u_{i+1}(t) \in (0, u^-)$ or if $u_i(t)$ and $u_{i+1}(t) \in (u^+, +\infty)$, then they cannot have a simultaneous phase transition.
	If both $u_i(t)$ and $u_{i+1}(t) \not \in (u^-, u^+)$, there cannot be an upwards phase transition at point $u_i(t)$ and a downward phase transition at point $u_{i+1}(t)$: otherwise, from \ref{inuguaglianza su} and \ref{inuguaglianza giu} we would have $\ns{\phi}(u_{i+1}(t)) > \ns{\phi}(u^-)$ or $\ns{\phi}(u_{i}(t)) < \ns{\phi}(u^+)$, against the necessity that $u_i(t) = u^-$ and $u_{i+1}(t) = u^+$.
	If $u_{i}(t) \in (u^-, u^+)$ and if $u_{i-1}(t)$ had an upwards phase transition, from \ref{inuguaglianza su} we would have $\ns{\phi}(u_{i-2}(t)) > \ns{\phi}(u^-)$, contradicting \ref{maximality}.
	If $u_i(t) \in (u^-, u^+)$ and if $u_{i+1}(t)$ had a downward phase transition, from \ref{inuguaglianza giu} we would have $\ns{\phi}(u_{i+2}(t)) < \ns{\phi}(u^+)$, against \ref{minimality}.
\end{proof}

Notice that Propositions \ref{riemann transitions} and \ref{interfaccia sottile} can be generalized to any piecewise S-continuous initial data taking values in $(0, u^-) \cup (u^+, +\infty)$: in this case, if the initial data has $n$ discontinuities, then $u_i(t) \in (u^-, u^+)$ for at most $n$ values of $i \leq N$.
In particular, if the initial data has finitely many discontinuities, then the dynamics of the system outside of the stable branches of $\phi$ is negligible.
In these cases, it could be argued by the above proposition that the phase transitions of $[u]$ trace a clockwise hysteresis loop, in agreement with the behaviour of two-phase solutions to \ref{mark} studied in \cite{evans survey, irreversibility, matete}.

We conclude our discussion of the Riemann problem with initial data \ref{riemann data} with a characterization of the asymptotic behaviour of the solution.

\begin{corollary}
	Let $u$ be the solution of system \ref{pb discreto} with initial data \ref{riemann data}.
	%If $\phi(l_1) < \phi(l_2)$, then the following are true:
	If $\ns{\phi}(p_l) > \ns{\phi}(p_r)$ then no phase transitions occur.
\end{corollary}
\begin{proof}
	It is a consequence of \ref{inuguaglianza su} and \ref{inuguaglianza giu} of Proposition \ref{riemann transitions} and of the fact that $\ns{\phi}(p_l) > \ns{\phi}(p_r)$ implies $u'_n(0) < 0$ and $u'_{n+1}(0) > 0$.
\end{proof}

\begin{corollary}
	Let $[u]$ be the grid solution of problem \ref{mark} with initial data \ref{riemann data}.
	Then $[u]$ converges to an asymptotically stable state that is either constant or Riemann-shaped.
\end{corollary}
\begin{proof}
	If no phase transitions occur, then the thesis is a consequence of Proposition \ref{limiti}.
	If phase transitions occur, this is a consequence of Proposition \ref{limiti} and of Proposition \ref{last stability}.
\end{proof}

\textbf{Acknowledgements}
We are grateful to Professor Vieri Benci for having suggested us the study of problem \ref{mark} with techniques from nonstandard analysis, and for the support he gave us during the writing of this paper.
Professors Imme van den Berg, Nigel Cutland, Todor T.\ Todorov and an anonymous referee also provided valuable feedback on the preliminary drafts of this paper.
The core of the paper was written while the author was affiliated to the University of Trento, Italy.

\end{document}